\documentclass[leqno,final]{siamltex}
\pdfoutput=1

\usepackage{graphicx}
\usepackage{amsmath,amstext,amssymb}
\newtheorem{remark}{Remark}[section]
\newcommand{\norm}[1]{\left\Vert#1\right\Vert}
\newcommand{\norml}[2]{\left\Vert#1\right\Vert_{L^2(#2)}}
\newcommand{\norme}[1]{\left\Vert{\hskip -3pt}\left\vert #1 \right\vert{\hskip -3pt}\right\Vert_{1,h}}
\newcommand{\abs}[1]{\bigl\vert#1\bigr\vert}
\newcommand{\pd}[1]{\left\langle #1\right\rangle}
\newcommand{\pdaj}[2]{\left\langle \left\{#1\right\}, \left[#2\right]\right\rangle_e}
\newcommand{\pdja}[2]{\left\langle \left[#1\right], \left\{#2\right\}\right\rangle_e}
\newcommand{\pdjj}[2]{\left\langle \left[#1\right], \left[#2\right]\right\rangle_e}
\newcommand{\set}[1]{\left\{#1\right\}}
\newcommand{\av}[1]{\left\{#1\right\}}
\newcommand{\jm}[1]{\left[#1\right]}

\newcommand{\tuh}{\tilde{u}_h}
\newcommand{\csta}{C_{\rm sta}}

\newcommand{\db}{\displaybreak[0]}
\newcommand{\nn}{\nonumber}

\newcommand{\al}{\alpha}
\newcommand{\be}{\beta}

\newcommand{\De}{\Delta}
\newcommand{\ep}{\varepsilon}
\newcommand{\ga}{\gamma}
\newcommand{\Ga}{\Gamma}
\newcommand{\la}{\lambda}

\newcommand{\na}{\nabla}
\newcommand{\om}{\omega}
\newcommand{\Om}{\Omega}
\newcommand{\pa}{\partial}
\newcommand{\pr}{\prime}
\newcommand{\si}{\sigma}

\renewcommand{\i}{{\rm\mathbf i}}

\DeclareMathOperator{\re}{{Re}}
\DeclareMathOperator{\im}{{Im}}

\newcommand{\br}{\mathbf{R}}
\newcommand{\bR}{\mathbf{R}}

\newcommand{\p}{\partial}

\newcommand{\Ome}{\Omega}
\newcommand{\nab}{\nabla}
\newcommand{\Del}{\Delta}

\newcommand{\cT}{\mathcal{T}}

\newcommand{\cE}{\mathcal{E}}

\newcommand{\loc}{{\rm loc}}

\newcommand{\Langle}{\left\langle}
\newcommand{\Rangle}{\right\rangle}

\def\jump#1{[#1]}
\def\Jump#1{\left[#1\right]}
\def\avrg#1{\{#1\}}

\title{Discontinuous Galerkin Methods for the Helmholtz Equation with Large
Wave Number}
\author{
Xiaobing Feng\thanks{Department of Mathematics, The University of
Tennessee, Knoxville, TN 37996, U.S.A.  ({\tt xfeng@math.utk.edu}).
The work of this author was partially supported by the NSF
grants DMS-0410266 and DMS-0710831.}
\and
Haijun Wu
\thanks{Department of Mathematics, Nanjing University, Jiangsu,
210093, P.R. China. ({\tt hjw@nju.edu.cn}). The work of this author was
partially supported by the national basic research program of China
under grant 2005CB321701 and by the program for the new century outstanding
talents in universities of China. Most part of this joint work was carried out
during the author's recent visit of the University of Tennessee, the author
would like to thank the host institution for its hospitality and
financial support of the visit.}}

\begin{document}

\maketitle


\setcounter{page}{1}

\large
\begin{abstract}
This paper develops and analyzes some interior penalty discontinuous Galerkin
methods using piecewise linear polynomials for the Helmholtz equation with
the first order absorbing boundary condition in the two and three dimensions.
It is proved that the proposed discontinuous Galerkin methods are stable (hence
well-posed) without any mesh constraint. For each fixed wave number $k$,
optimal order (with respect to $h$) error estimate in the broken $H^1$-norm and
sub-optimal order estimate in the $L^2$-norm are derived without any mesh 
constraint. The latter estimate improves to optimal order when the mesh size $h$ 
is restricted to the preasymptotic regime  (i.e., $k^2 h \gtrsim 1$). 
Numerical experiments are also presented to gauge the theoretical result and
to numerically examine the pollution effect (with respect to $k$) in the error
bounds. The novelties of the proposed interior penalty discontinuous 
Galerkin methods include: first, the methods penalize not only 
the jumps of the function values
across the element edges but also the jumps of the normal and tangential
derivatives; second, the penalty parameters are taken as complex numbers
of positive imaginary parts so essentially and practically no constraint is
imposed on the penalty parameters.
Since the Helmholtz problem is a non-Hermitian and indefinite linear
problem, as expected, the crucial and the most difficult
part of the whole analysis is to establish the stability estimates (i.e.,
a priori estimates) for the numerical solutions. To the end, the cruxes
of our analysis are to establish and to make use of a local
version of the Rellich identity (for the Laplacian) and to mimic
the stability analysis for the PDE solutions given in
\cite{cummings00,Cummings_Feng06,hetmaniuk07}.
\end{abstract}

\begin{keywords}
Helmholtz equation, time harmonic waves, absorbing boundary conditions,
discontinuous Galerkin methods, error estimates
\end{keywords}

\begin{AMS}
65N12, 
65N15, 
65N30, 
78A40  
\end{AMS}


\section{Introduction}\label{sec-1}
Wave is ubiquitous, it arises in many branches of science, engineering
and industry (cf. \cite{ck92,ked08} and the references therein).
It is significant to geoscience, petroleum engineering,
telecommunication, and defense industry. Mathematically, wave
propagation problems are described by hyperbolic partial differential equations,
and the progress of solving wave-related application problems has largely
depended on the progress of developing effective methods and algorithms to
solve their governing partial differential equations. Among many wave-related
application problems, those dealing with high frequencies (or large wave numbers)
wave propagation are most difficult to solve numerically because they
are strongly indefinite and non-Hermitian and their solutions
are very oscillatory. These properties in turn make it very
difficult to construct stable numerical schemes under practical
mesh constraints. Furthermore, high frequency (or large wave number)
requires to use very fine meshes in order to resolve highly oscillatory waves,
and the use of fine meshes inevitably gives rise huge, strongly indefinite,
ill-conditioned, and non-Hermitian (algebraic) systems to solve. All these
difficulties associated with high frequency (large wave number) wave
computation still remain to be resolved and are not mathematically well
understood in high dimensions (cf. \cite{zienkiewicz00}), although
considerable amount of progresses have been made in the past thirty
years (cf. \cite{er03,ihlenburg98,ked08} and the references therein).

The simplest prototype wave scattering problem is the following
acoustic scattering problem (with time dependence $e^{\i\om t}$):
\begin{alignat}{2}\label{e0.1}
-\Del u- k^2 u &= f \qquad &&\mbox{in } \bR^d\setminus D,\\
u &=0&&\mbox{on } \p D, \label{e0.2} \\
\sqrt{r}\Bigl( \frac{\p (u-u^{\rm inc})}{\p r} + \i k (u-u^{\rm inc})\Bigr) &\rightarrow 0 &&\mbox{as }
r=|x|\rightarrow \infty, \label{e0.3}
\end{alignat}
where $D\subset\bR^d \,(d=2,3)$, a bounded Lipschitz domain, denotes the
scatterer and
$\i=\sqrt{-1}$ denotes the imaginary unit. $k\in\mathbf{R}$ is a given
positive (large) number and known as the wave number. $u^{\rm inc}$ is the incident wave.
Equation \eqref{e0.1} is the well-known Helmholtz equation and condition
\eqref{e0.3} is the Sommerfeld radiation condition at the infinity.
Boundary condition \eqref{e0.2} implies that the scatterer is sound-soft.

To compute the solution of the above problem, due to (finite) memory and
speed limitations of computers, one needs first to formulate the problem
as a finite domain problem. Two major
approaches have been developed for the task in the past thirty years.
The first approach is {\em boundary integral methods} (cf. \cite{yu02}
and the references therein) and the other one
is {\em artificial boundary condition methods} (cf. \cite{em79,berenger94}
and the references therein). In the case of boundary integral methods,
one converts the original differential equation into a (complicate) boundary
integral equation on the boundary $\pa D$ of the scatterer $D$.
Clearly, the trade-off is that the original simple differential equation
could not be retained in the conversion. On the other hand, artificial
boundary condition methods solve the given differential equation on
a truncated computational domain by imposing suitable artificial
boundary conditions on the outer boundary of the computational domain.
Artificial boundary condition methods can be
divided into two groups. One group of the methods use sharp artificial
boundaries (i.e., the boundary has zero width), appropriate boundary
conditions, which are called ``{\em absorbing boundary conditions}", then are
imposed on the boundaries (cf. \cite{em79,feng99}). The second group of
artificial boundary condition methods allow the artificial boundaries to have
non-zero width, such fatten boundaries are called absorbing layers,
where the artificial boundary conditions are usually constructed in
the form of differential equations which replace the original wave
equations in the absorbing layers.  The methods of this second
group are called ``{\em PML (perfectly matched layer) methods}"
(cf. \cite{berenger94}).  In this paper, we shall adopt the absorbing boundary
condition approach for problem \eqref{e0.1}--\eqref{e0.3}. Extension of the
results of this paper to the PML formulations will be given elsewhere.

Specifically, in this paper we consider the following Helmholtz problem:
\begin{alignat}{2}
-\Del u - k^2 u &=f  &&\qquad\mbox{in  } \Ome:=\Ome_1\setminus D,\label{e1.1}\\
u &=0 &&\qquad\mbox{on } \Ga_D,\label{e1.3}\\
\frac{\pa u}{\pa n_\Om} +\i k u &=g &&\qquad\mbox{on } \Ga_R,\label{e1.2}
\end{alignat}
where $(D\subset)\, \Ome_1\subset \br^d,\, d=2,3$ is a polygonal/polyhedral
domain, which is often taken as a $d$-rectangle in applications.
$\Ga_R:=\pa\Ome_1, \Ga_D=\pa D$, hence, $\pa\Ome=\Ga_R\cup \Ga_D$. $n_\Om$
denotes the unit outward normal
to $\pa\Om$. The Robin boundary condition \eqref{e1.2} is known as the
first order absorbing boundary condition (cf. \cite{em79}).
We remark that the case $D=\emptyset$ also arises in applications
either as a consequence of frequency domain treatment of waves or
when time-harmonic solutions of the scalar wave equation are sought
(cf. \cite{dssb93,dss94}).

For many years, the finite element method (and other type methods)
has been widely used to discretize the Helmholtz equation (\ref{e1.1})
with various types of boundary conditions (cf. \cite{amm06,aldr06,
ak79,aw80,bao95,cw03,cm87b,cummings00,dss94,er03,goldstein81,ib95a,ihlenburg98,
ked08,op98} and the reference therein).
It is well known that in every
coordinate direction, one must put some minimal number of grid points
in each wave length $\ell=2\pi/k$ in order to resolve the wave, that is,
the mesh size $h$ must satisfy the constraint $hk\lesssim 1$. In practice,
$6-10$ grid points are used in a wave length, which is often referred
as the ``rule of thumb". However, this ``rule of thumb" was proved
rigorously not long ago by Babu{\v{s}}ka {\em et al} \cite{ib95a}
only in the one-dimensional case (called the {\em preasymptotic
error analysis}). The main difficulty of the analysis is caused by
the strong indefiniteness of the Helmholtz equation which in turn
makes it hard to establish stability estimates for the finite element
solution under the ``rule of thumb" mesh constraint.
In \cite{ib95a}, Babu{\v{s}}ka {\em et al} also showed that
the $H^1$-error bound for the finite element solution contains a pollution term
that is related to the loss of stability with large wave numbers. Later,
Babu{\v{s}}ka {\em et al} addressed the question whether it is possible
to reduce the pollution effect in a series of papers
(cf. Chapter 4 of \cite{ihlenburg98} and the reference therein).
It should be noted that under the stronger mesh condition 
that $k^2h$ is sufficiently small, optimal (with respect to $h$) 
and quasi-optimal (with respect to $k$)
error estimates for finite element approximations of the
Helmholtz problem were established early by Aziz and Kellogg in \cite{ak79}
and Douglas, {\em et al} in \cite{dssb93, dss94}
using the so-called Schatz argument \cite{schatz74} (also see
Chapter 5 of \cite{bs94}), and a similar result was also obtained in
\cite{goldstein81} using an operator perturbation argument. 

The work of \cite{ak79,dssb93,ib95a} shows that in the $1$-d case, due to the
pollution effect, the finite element solution for the Helmholtz
problem (\ref{e1.1})-(\ref{e1.3}) deteriorates as the wave number becomes
large if the practical mesh condition $kh\lesssim 1$ is used. The situation
in the high dimensions is expected to be the same (at least not better) although,
to the best of our knowledge, no such a rigorous analysis is
known in the literature. The detailed analysis of \cite{dssb93,ib95a}
also shows that the pollution effect is inherent in the finite element method
and is caused by the deterioration of stability of the Helmholtz operator
as the wave number $k$ becomes large.
In order to minimize or eliminate (if possible) the pollution and to obtain
more stable and more accurate approximate solutions for Helmholtz-type
problems with large wave numbers, various nonstandard and generalized Galerkin
methods have been proposed lately in the literature. These methods can be categorized
into two groups. The first group of methods use nonstandard or stabilized
discrete variational forms to approximate the Helmholtz operator so that the
resulted discrete problems have better stability properties. Methods in this
group include Galerkin-least-squares finite element methods
\cite{chang90,hh92}, quasi-stabilized finite element methods \cite{bs97},
and discontinuous Galerkin methods \cite{amm06,ce06,perugia07}.
The second group of methods abandon the use of piecewise polynomial
trial and test functions and replace them by global polynomials or non-polynomial
functions. Methods in this group include spectral methods \cite{sw07},
generalized Galerkin/finite element methods \cite{melenk95,bbo04},
partition of unity finite element methods \cite{mb96},
and meshless methods \cite{bbo04}.
We also note that another very different and intensively studied
approach for high frequency wave computation is {\em geometrical optics}, which
studies asymptotic (nonlinear) approximations of the Helmholtz equation
obtained when the frequency (or wave number) tends infinity. We refer the
reader to \cite{er03} and the references therein for some
recent developments in geometrical optics and its variants.

The goal of this paper is to develop some interior penalty
discontinuous Galerkin (IPDG) methods for problem
\eqref{e1.1}--\eqref{e1.3} in high dimensions.  The focus of the paper
is to establish the rigorous stability and error analysis, in particular,
the preasymptotic error analysis.  For the ease of presentation and to
better present ideas, we confine ourselves to only consider the case of
linear element in this paper. Such a restriction is also due to the
consideration that we shall present $hp$-discontinuous Galerkin methods
for problem \eqref{e1.1}--\eqref{e1.3} in a forth coming paper \cite{fw08b}
which extends the work of this paper to high order elements.
Compared with existing DG methods for the Helmholtz equation in the
literature, the novelties of our interior penalty discontinuous Galerkin methods
include the following: First, our mesh-dependent sesquilinear forms penalize not
only the jumps of the function values across the element edges/faces but also the
jumps of the normal and tangential derivatives. Recall that penalizing
the jumps of the normal (and tangential) derivatives helps but is not essential
for the success of IPDG methods in the case of coercive elliptic problems
(cf. \cite{arnold82,dd76}), however, it contributes critically
to the stability of the IPDG methods of this paper.
Second, a small but vitally important idea of this paper is
to take the penalty parameters as complex numbers of positive imaginary parts.
This idea also contributes critically
to the stability of the IPDG methods of this paper. As a result,
essentially and practically no constraint is imposed on the penalty parameters.
Since the Helmholtz problem is a non-Hermitian and an indefinite linear
problem, as expected, the crucial and the most difficult
part of the whole analysis is to establish the stability estimates (i.e.,
a priori estimates) for the numerical solutions. To the end, the cruxes
of our analysis are to establish and to make use of a local
version of the Rellich identity (for the Laplacian) and to mimic
the stability analysis for the PDE solutions given
in \cite{cummings00,Cummings_Feng06,hetmaniuk07}.
Suppose $\Om_1$ is star-shaped with respect to a point $x_{\Om_1}$. 
The key idea is to use the special test function $\nab u_h\cdot (x-x_{\Om_1})$
(defined element-wise), which is a valid candidate for any IPDG method. 
We remark that the same technique was successfully employed by 
Shen and Wang in \cite{sw07} to carry out the stability and error 
analysis for the spectral Galerkin
approximation of the Helmholtz problem.

In the past fifteen years, DG methods have received a lot attentions
and undergone intensive studies by many people. As is well known now,
DG methods have several advantages over other types of numerical methods.
For example, the trial and test spaces are very easy to construct, they can
naturally handle inhomogeneous boundary conditions and curved boundaries;
they also allow the use of highly nonuniform and unstructured meshes,
and have built-in parallelism which permits coarse-grain parallelization.
In addition, the fact that the mass matrices are block diagonal is an
attractive feature in the context of time-dependent problems, especially
if explicit time discretizations are used. We refer to
\cite{arnold82,abcm01,baker77,CKS00,CS98,dd76,fk07,ob00,rwg99,w78}
and the references therein for a detailed account on DG methods for
coercive elliptic and parabolic problems.
In addition to the advantages listed above, the results of this paper also
demonstrate the flexibility and effectiveness of DG methods for
strongly indefinite problems, which was not well understood before.

The remainder of this paper is organized as follows. In Section \ref{sec-2a},
notation and some preliminaries are described and cited. In particular,
the sharp (with respect to $k$) stability constant estimates of
\cite{Cummings_Feng06} for the solution of problem \eqref{e1.1}--\eqref{e1.3}
in high dimensions were recalled. These estimates are critical
for obtaining explicit dependence of the error bounds on the
wave number $k$. In Section~\ref{sec-2},  the IPDG methods of this paper
are formulated. Both symmetric and non-symmetric IPDG methods are constructed.
However, since the Helmholtz equation and its solution are
complex-valued, the non-symmetric terms in the IPDG sesquilinear form
do not cancel each other when two arguments of the form
are taken to be the same function. Instead, their difference is
a pure imaginary quantity. This is a main difference between
non-symmetric IPDG for coercive elliptic problems and for
indefinite Helmholtz type problems. As a result, the penalty parameters
need to be chosen as complex numbers with positive imaginary parts
to ensure the stability in both symmetric
and non-symmetric IPDG methods. Section~\ref{sec-sta} devotes
to stability analysis for the IPDG methods proposed in Section~\ref{sec-2}.
It is proved that the proposed IPDG methods are stable (hence
well-posed) without any mesh constraint.
In Section~\ref{sec-err}, using the stability result of Section~\ref{sec-sta}
we derive optimal order (with respect to $h$) 
error estimate in the broken $H^1$-norm and
sub-optimal order estimate in the $L^2$-norm without any mesh constraint.
The latter estimate improves to optimal order when the mesh size $h$ 
is restricted to the preasymptotic regime (i.e., $k^2 h \gtrsim 1$).    
In particular, for appropriately chosen penalty
parameters, it is shown that the error in the broken $H^1$-norm is bounded by
$\widetilde{C}_1kh + \widetilde{C}_2k^{8/3}h^{4/3}$  if $kh\lesssim 1$.
Numerical tests in Section~\ref{sec-num} suggest that the error in the broken
$H^1$-norm may have a better bound $\widetilde{C}_1kh + \widetilde{C}_2k^{3}h^{2}$
and it is possible to tune the penalty parameters to significantly reduce the
pollution error. We note that in the case $k^2h$ is sufficiently small,
optimal order (with respect to $h$) error estimate in the broken 
$H^1$-norm can be derived by using the Schatz argument as done 
in \cite{perugia07,ak79, dssb93, dss94}.
 In Section \ref{sec-num}, we present some numerical experiments
to gauge our theoretical error estimates, to numerically examine the pollution
effect in the error bounds, and to test the performance of the proposed IPDG methods.

\section{Notation and Preliminaries} \label{sec-2a}
The standard space, norm and inner product notation
are adopted. Their definitions can be found in \cite{bs94,ciarlet78,baker77}.
In particular, $(\cdot,\cdot)_Q$ and $\langle \cdot,\cdot\rangle_\Sigma$
for $\Sigma\subset \pa Q$ denote the $L^2$-inner product
on {\em complex-valued} $L^2(Q)$ and $L^2(\Sigma)$
spaces, respectively. $(\cdot,\cdot):=(\cdot,\cdot)_\Ome$
and $\langle \cdot,\cdot\rangle:=\langle \cdot,\cdot\rangle_{\p\Ome}$. Let
\[
H_{\Ga_D}^1(\Om):=\set{u\in H^1(\Om): u=0 \text{ on }\Ga_D}.
\]
Throughout the paper, $C$ is used to denote a generic positive constant
which is independent of $h$ and $k$. We also use the shorthand
notation $A\lesssim B$ and $B\gtrsim A$ for the
inequality $A\leq C B$ and $B\geq CA$. $A\simeq B$ is a shorthand 
notation for the statement $A\lesssim B$ and $B\lesssim A$.

We now recall the definition of star-shaped domains.

\begin{definition}\label{def1}
$Q\subset \bR^d$ is said to be a {\em star-shaped} domain with respect
to $x_Q\in Q$ if there exists a nonnegative constant $c_Q$ such that
\begin{equation}\label{estar}
(x-x_Q)\cdot n_Q\ge c_Q \qquad \forall x\in\pa Q.
\end{equation}
$Q\subset \bR^d$ is said to be {\em strictly star-shaped} if $c_Q$ is positive.
\end{definition}

Throughout this paper, we assume that $\Ome_1$ is a strictly star-shaped domain.
In practice, $\Ome_1$ is often taken as a $d$-rectangle, which
trivially is a strictly star-shaped domain. We also assume the scatterer $D$ is a
star-shaped domain with respect to
the same point $x_{\Ome_1}$ as $\Ome_1$ does.  This then implies that
$x_{\Ome_1}\in D\subset \Ome_1$. More precisely, we assume that there exist constants $c_{\Om_1}>0$ and $c_D\ge 0$ such that
\begin{equation}\label{estar1}
(x-x_{\Om_1})\cdot n_\Om \ge c_{\Om_1} \quad \forall x\in\Ga_R\quad\text{ and }\quad(x-x_{\Om_1})\cdot n_D \ge c_{D} \quad \forall x\in\Ga_D.
\end{equation}
Here $n_\Om$ and $n_D$ are the unit outward normals to the boundaries
of $\Om$ and $D$, respectively.

 Under these assumptions the following
stability estimates for problem \eqref{e1.1}--\eqref{e1.3}
were proved in \cite{cummings00,Cummings_Feng06,hetmaniuk07}.

\begin{theorem}\label{stability}
Suppose $\Om_1\subset \bR^d$ is a strictly star-shaped domain and $D\subset\Om_1$
is a star-shaped domain. Then the solution
$u$ to problem \eqref{e1.1}--\eqref{e1.3} satisfies
\begin{eqnarray}\label{e2.1}
\|u\|_{H^j(\Ome)} \lesssim \Bigl(\frac{1}k+ k^{j-1} \Bigr)
\bigl( \|f\|_{L^2(\Ome)} + \|g\|_{L^2(\Ga_R)} \bigr)
\end{eqnarray}
for $j=0, 1$ if $u\in H^{3/2+\ep}(\Om)$ for some $\ep>0$.  (2.3) also holds for $j=2$ if $u\in H^2(\Om).$\end{theorem}

\section{Formulation of discontinuous Galerkin methods}\label{sec-2}
To formulate our IPDG methods, we first need to introduce some notation.
Let $\cT_h$ be a family of triangulations of the domain $\Ome:=\Ome_1\setminus D$
parameterized by $h>0$. For any triangle/tetrahedron $K\in \cT_h$, we define $h_K:=\mbox{diam}(K)$.
Similarly, for each edge/face $e$ of $K\in \cT_h$, define $h_e:=\mbox{diam}(e)$.
We assume that the elements of $\cT_h$ satisfy the minimal angle condition. We define
\begin{eqnarray*}
\cE_h^I&:=& \mbox{ set of all interior edges/faces of $\cT_h$} ,\\
\cE_h^R&:=& \mbox{ set of all boundary edges/faces of $\cT_h$ on $\Ga_R$} ,\\
\cE_h^D&:=& \mbox{ set of all boundary edges/faces of $\cT_h$ on $\Ga_D$} ,\\
\cE_h^{RD}&:=& \cE_h^R\cup\cE_h^D= \mbox{ set of all boundary edges/faces of $\cT_h$} ,\\
\cE_h^{ID}&:=&\cE_h^I\cup\cE_h^D= \mbox{ set of all edges/faces of $\cT_h$ except those on $\Ga_R$}.
\end{eqnarray*}
We also define the jump $\jump{v}$ of $v$ on an interior edge/face
$e=\p K\cap \p K^\pr$ as
\[
\jump{v}|_{e}:=\left\{\begin{array}{ll}
       v|_{K}-v|_{K^\pr}, &\quad\mbox{if the global label of $K$ is bigger},\\
       v|_{K^\pr}-v|_{K}, &\quad\mbox{if the global label of $K^\pr$ is bigger}.
\end{array} \right.
\]
If $e\in\cE_h^{D}$, set $\jump{v}|_{e}=v|_{e}$. The following convention is
adopted in this paper
\[
\avrg{v}|_{e}:=\frac12\bigl( v|_{K}+ v|_{K^\pr} \bigr)\qquad \mbox{if }
e=\p K\cap \p K^\pr.
\]
If $e\in\cE_h^{D}$, set $\avrg{v}|_{e}=v|_{e}$. For every
$e=\p K\cap \p K^\pr\in\cE_h^I$, let $n_e$ be the unit outward normal
to edge/face $e$ of the element $K$ if the global label of $K$ is bigger
and of the element $K^\pr$ if the other way around. For every $e\in\cE_h^{RD}$, let $n_e=n_\Om$ the unit outward normal to $\pa\Om$.

Now we define the ``energy" space $E$ and the sesquilinear
form $a_h(\cdot,\cdot)$ on $E\times E$ as follows:
\begin{align}
E&:=\prod_{K\in\cT_h} H^2(K), \nonumber \\
\label{eah}
a_h(u,v)&:=b_h(u,v)+ \i \bigl( J_0(u,v)+J_1(u,v)+L_1(u,v) \bigr) \qquad\forall\, u, v\in E,
\end{align}
where
\begin{align}
b_h(u,v):=&\sum_{K\in\cT_h} (\nab u,\nab v)_K -\sum_{e\in\cE_h^{ID}}
\left( \pdaj{\frac{\pa u}{\pa n_e}}{v}+\si\pdja{u}{\frac{\pa v}{\pa n_e}}\right),\label{ebh}\db\\
J_0(u,v):=&\sum_{e\in\cE_h^{ID}}\frac{\ga_{0,e}}{h_e} \pdjj{u}{v},\label{eJ0}\db\\
J_1(u,v):=&\sum_{e\in\cE_h^{I}}\ga_{1,e} h_e \pdjj{\frac{\pa u}{\pa n_e}}{\frac{\pa v}{\pa n_e}},
\label{eJ1}\db\\
L_1(u,v):=&\sum_{e\in\cE_h^{ID}}\sum_{j=1}^{d-1}\frac{\be_{1,e}}{h_e} \pdjj{\frac{\pa u}{\pa \tau_e^j}}
{\frac{\pa v}{\pa \tau_e^j}},\label{eL1}
\end{align}
and $\si$ is an $h$-independent real number. $\ga_{0,e}, \ga_{1,e}$, and $\be_{1,e}$ are
nonnegative numbers to be specified later. $\{\tau^j_e\}_{j=1}^{d-1}$
denote an orthogonal coordinate frame on the edge/face $e\in \cE_h$, and
$\frac{\pa u}{\pa \tau_e^j}:=\nab u\cdot \tau_e^j$ stands for the
tangential derivative of $u$ in the direction $\tau_e^j$.

\begin{remark}
(a) Clearly, $a_h(\cdot,\cdot)$ is a consistent discretization for
$-\Del$ since $(-\Del u,v) = a_h(u,v)$ for all $u\in H_{\Ga_D}^1(\Om)\cap H^2(\Ome)$ and $v\in E$.

(b) If we regard  $a_h(\cdot,\cdot)$ as a bilinear form on the
subspace of real valued functions in $H_{\Ga_D}^1(\Om)$, then $a_h(\cdot,\cdot)$
is symmetric when $\si=1$ and is non-symmetric when $\si\neq 1$. In particular, $\si =-1$
would correspond to the non-symmetric IPDG method studied in \cite{ob00,rwg99}
for coercive elliptic problems. In this paper, for the ease of presentation,
we only consider the case $\si=1$, nevertheless the main results of
the paper can also be extended to the case  $\si\neq 1$.

(c) The terms in $\i\bigl( J_0(u,v)+J_1(u,v)+L_1(u,v)\bigr)$ are
so-called penalty terms.

(d) The penalty parameters in $\i\bigl( J_0(u,v)+J_1(u,v)+L_1(u,v) \bigr)$ are
$\i\gamma_{0,e}, \i\gamma_{1,e}$ and $\i\beta_{1,e}$, respectively.  So they are
pure imaginary numbers with positive imaginary parts. It turns out that if any
of them is replaced by a complex number with  positive
imaginary part, the ideas of the paper still apply.  Here we set their
real parts to be zero partly because the terms from real parts do not help much
(and do not cause any problem either) in our theoretical analysis and partly for
the ease of presentation. On the other hand, our numerical experiments
in Section \ref{sec-6.5} indicate that using penalty parameters
with nonzero real parts helps to reduce the pollution effect in the error.

(e) Penalizing the jumps of normal derivatives (i.e., the $J_1$ term above)
for second order PDEs was used early by Douglas and Dupont \cite{dd76} in
the context of $C^0$ finite element methods, by Baker \cite{baker77}
(with a different weighting, also see \cite{fk07}) for fourth order PDEs,
and by Arnold \cite{arnold82} in the context of IPDG methods for second order
parabolic PDEs.  Arnold \cite{arnold82} also proposed and analyzed IPDG
methods which penalize higher order normal derivatives. Note that we do
not introduce boundary terms for $e\in\cE_h^D$ in $J_1$ to ensure the
consistency of $a_h(\cdot,\cdot)$ with $-\Del$.

On the other hand, the idea of penalizing the jumps of tangential derivatives (i.e.,
the $L_1$ term above) seems is new. Later we will show that without $L_1$ term and $J_1$ term in
$a_h(\cdot,\cdot)$ the IPDG methods of this paper are still stable and convergent
but under a stinger mesh constraint, see Section \ref{sec-sta} and \ref{sec-err}.

(f) In this paper we consider the scattering problem with time dependence $e^{\i\om t}$, that is, the signs before $\i$'s in the Sommerfeld radiation condition \eqref{e0.3} and its first order approximation \eqref{e1.2} are positive. If we consider the scattering problem with time dependence $e^{-\i\om t}$, that is, the signs before $\i$'s in  \eqref{e0.3} and  \eqref{e1.2} are negative, then the penalty parameters should be complex numbers with  negative imaginary parts.
\end{remark}

Next, we introduce the following semi-norms/norms on the space $E$:
\begin{align}
\abs{v}_{1,h}:=&\Big(\sum_{K\in\cT_h} \norml{\nab v}{K}^2\Big)^{\frac12},\label{e2.5b}\db\\
\norm{v}_{1,h}:=&\left( \abs{v}_{1,h}^2+\sum_{e\in\cE_h^{I}}\ga_{1,e} h_e \norml{\jm{\frac{\pa v}{\pa n_e}}}{e}^2\right.
 \label{e2.5}\\
&\hskip 0.5in  \left.+\sum_{e\in\cE_h^{ID}} \left( \frac{\ga_{0,e}}{h_e}\norml{\jm{v}}{e}^2
+\sum_{j=1}^{d-1}\frac{\be_{1,e}}{h_e}\norml{\jm{\frac{\pa v}{\pa \tau_e^j}}}{e}^2\right)\right)^{\frac12},
\nn \db\\
\norme{v}:=&\left(\norm{v}_{1,h}^2+\sum_{e\in\cE_h^{ID}}
\frac{h_e}{\ga_{0,e}}\norml{\av{\frac{\pa v}{\pa n_e}}}{e}^2 \right)^{\frac12}.\label{e2.5a}
\end{align}
It is easy to see that $\norm{\cdot}_{1,h}$ and $\norme{\cdot}$ are norms on $E$ if
$\p D\neq \emptyset$ but only semi-norms if $\p D= \emptyset$.

Clearly, the sesquilinear form $a_h(\cdot,\cdot)$ with $\sigma=1$ satisfies: For any $v\in E$
\begin{align}
\re a_h(v,v)&=\abs{v}_{1,h}^2-2\re \sum_{e\in\cE_h^{ID}}\pdaj{\frac{\pa v}{\pa n_e}}{v},
\label{ea1}\\
\im a_h(v,v)&=J_0(v,v)+J_1(v,v)+L_1(v,v). \label{ea2}
\end{align}

With the help of the sesquilinear form $a_h(\cdot,\cdot)$ we now
introduce the following weak formulation for \eqref{e1.1}--\eqref{e1.3}:
Find $u\in E\cap H_{\Ga_D}^1(\Om)\cap H^2_{\loc}(\Ome)$ such that
\begin{equation}
a_h(u,v) - k^2(u,v) +\i k \langle u,v\rangle_{\Ga_R}
=(f,v)+\pd{g, v}_{\Ga_R},\qquad\forall v\in E\cap H_{\Ga_D}^1(\Om)\cap H^2_{\loc}(\Ome).
\label{2.6}
\end{equation}
The above formulation is consistent with the boundary value problem
\eqref{e1.1}--\eqref{e1.2} because $a_h(\cdot,\cdot)$ is consistent
with $-\Del$. It is clear that, if $u\in H^2(\Om)$ is the solution of \eqref{e1.1}--\eqref{e1.3}, then \eqref{2.6} holds for all $v\in E.$

For any $K\in \cT_h$, let $P_1(K)$ denote the set of all linear
polynomials on $K$. We define our IPDG approximation space $V^h$ as
\[
V^h:=\prod_{K\in \cT_h} P_1(K).
\]
Clearly, $V^h\subset E\subset L^2(\Ome)$. But $V^h\not\subset H^1(\Om)$.

We are now ready to define our IPDG methods based on the weak formulation
\eqref{2.6}: Find $u_h\in V^h$ such that
\begin{equation}\label{edg}
a_h(u_h,v_h) - k^2(u_h,v_h) +\i k \langle u_h,v_h\rangle_{\Ga_R}
=(f, v_h)+\pd{g, v_h}_{\Ga_R}  \qquad\forall v_h\in V^h.
\end{equation}

In the next two sections, we shall study the stability and
error analysis for the above IPDG methods. Especially, we are interested
in knowing how the stability constants and error constants depend on the wave
number $k$ (and mesh size $h$, of course) and what are the ``optimal" relationship
between mesh size $h$ and the wave number $k$.

\section{Stability estimates}\label{sec-sta}
Since the Helmholtz operator is not a coercive elliptic operator,
the stability estimates given in Theorem \ref{stability} for the solution
of problem \eqref{e1.1}--\eqref{e1.2} is far from trivial. We refer
the reader to \cite{cummings00,Cummings_Feng06,hetmaniuk07} for a detailed exposition
in this direction. This difficulty is certainly inherited by any numerical
discretization of problem \eqref{e1.1}--\eqref{e1.2}. In fact, the situation
usually is worse in the discrete case because piecewise polynomials
(or piecewise smooth functions) are rigid, they are not as
flexible as the PDE trial functions. On the other hand, DG approximation
functions are much more flexible than Lagrange finite element functions
because they do not have any continuity constraint, instead, the
continuity of the numerical solutions is enforced weakly through
mesh-dependent bilinear or sesquilinear or nonlinear forms.

The goal of this section is to derive stability estimates (or a priori estimates)
for scheme \eqref{edg}.  To the end, momentarily, we assume solution $u_h$ to \eqref{edg}
exists and will revisit the existence and uniqueness issues later at the end of
the section. We like to note that because its strong indefiniteness,
unlike in the case of coercive elliptic and parabolic problems
(cf. \cite{arnold82,abcm01,baker77,dd76,fk07,ob00,rwg99,w78}),
the well-posedness of scheme \eqref{edg} is far from obvious
under practical mesh constraints.

To derive stability estimates for scheme \eqref{edg}, our approach is to mimic
the stability analysis for the Helmholtz problem \eqref{e1.1}--\eqref{e1.2}
given in \cite{cummings00,Cummings_Feng06,hetmaniuk07}. It turns out that this approach
indeed works for scheme \eqref{edg} although the analysis is
more delicate and complicate than that for the differential problem.
The key ingredients of our analysis are to use a special test function
$v_h=\al\cdot\na u_h$ (defined element-wise) with $\al(x):=x-x_{\Ome_1}$
in \eqref{edg} and to use the Rellich identity (cf. \cite{Cummings_Feng06} and
below) on each element.

Our first lemma of this section establishes three integral identities
which play an important role in our analysis.

\begin{lemma}\label{lem2.1}
Let $\al(x):=x-x_{\Ome_1}$, $v\in E$, $K\in\cT_h$ and $e\in \cE_h^{ID}$.
Then there hold
\begin{align}
&d\norml{v}{K}^2+2\re(v,\al\cdot\na v)_K
=\int_{\pa K}\al\cdot n_K\abs{v}^2,\label{eid1}\\
&(d-2)\norml{\na v}{K}^2+2\re\big(\na v,\na(\al\cdot\na v)\big)_K
=\int_{\pa K}\al\cdot n_K\abs{\na v}^2,\label{eid2}\\
&\pdaj{\frac{\pa v}{\pa n_e}}{\al\cdot\na v}
-\pd{\al\cdot n_e\av{\na v},\jm{\na v}}_e  \label{eid3}\\
&\hskip 1.35in=\sum_{j=1}^{d-1}\int_e\left(\al\cdot\tau_e^j\av{\frac{\pa v}{\pa n_e}}
-\al\cdot n_e\av{\frac{\pa v}{\pa \tau_e^j}}\right)\frac{\pa\jm{\overline{v}}}{\pa\tau_e^j},
\nn
\end{align}
where $x_{\Ome_1}$ denotes the point in the star-shaped domain definition for
$\Ome_1$ (see  \eqref{estar1}). Also note that in \eqref{eid1} and \eqref{eid2},
we omit the sign $\mbox{ds}$ in the integrals. We shall adopt this omission consistently
throughout this paper to save space.
\end{lemma}

\begin{proof}
It is easy to verify by direct computations the following differential identities on $K$:
\begin{align*}
\mathrm{div} (\al(x)) &\equiv d,\\
\mathrm{div}(\al v\overline{v})&=d v\overline{v}+(\al\cdot\na v)\overline{v}
+v(\al\cdot\na \overline{v}),\\
\mathrm{div}(\al (\na v\cdot\na\overline{v}))&=(d-2) \na v\cdot\na\overline{v}
+\na(\al\cdot\na v)\cdot\na\overline{v}+\na v\cdot\na(\al\cdot\na \overline{v}).\\
\end{align*}
\eqref{eid1} and \eqref{eid2} then follows immediately from integrating the
above second and third identities over $K$.

To prove identity \eqref{eid3}, from the representations
\begin{align*}
\na v =\frac{\pa v}{\pa n_e} n_e+\sum_{j=1}^{d-1}\frac{\pa v}{\pa\tau_e^j}\tau_e^j,
\hskip 0.5in
\al\cdot\na v =\frac{\pa v}{\pa n_e} \al\cdot n_e
+\sum_{j=1}^{d-1}\frac{\pa v}{\pa\tau_e^j}\al\cdot \tau_e^j,
\end{align*}
we have
\begin{align*}
&\pdaj{\frac{\pa v}{\pa n_e}}{\al\cdot\na v}-\pd{\al\cdot n_e\av{\na v},\jm{\na v}}_e\\
&\hskip 0.6in
=\int_e\left(\al\cdot n_e\av{\frac{\pa v}{\pa n_e}}\jm{\frac{\pa \overline{v}}{\pa n_e}}
+\sum_{j=1}^{d-1}\al\cdot\tau_e^j\av{\frac{\pa v}{\pa n_e}}
\jm{\frac{\pa\overline{v}}{\pa\tau_e^j}}\right)\\
&\hskip 1in
-\int_e\left(\al\cdot n_e\av{\frac{\pa v}{\pa n_e}}\jm{\frac{\pa \overline{v}}{\pa n_e}}
+\al\cdot n_e\sum_{j=1}^{d-1}\av{\frac{\pa v}{\pa \tau_e^j}}
\jm{\frac{\pa\overline{v}}{\pa\tau_e^j}}\right)\\
&\hskip 0.6in
=\sum_{j=1}^{d-1}\int_e\left(\al\cdot\tau_e^j\av{\frac{\pa v}{\pa n_e}}
-\al\cdot n_e\av{\frac{\pa v}{\pa \tau_e^j}}\right)\frac{\pa\jm{\overline{v}}}{\pa\tau_e^j},
\end{align*}
which completes the proof of the lemma.
\end{proof}

\begin{remark}
The identity \eqref{eid2} can be viewed as a local version of the Rellich identity for
the Laplacian $\Del$ (cf. \cite{cummings00,Cummings_Feng06}).  Since $V^h\subset E$,
hence, \eqref{eid1}--\eqref{eid3} also hold for any function $v=v_h\in V^h$.
\end{remark}

Now, taking $v_h=u_h$ in \eqref{edg} yields
\begin{equation}\label{euh}
a_h(u_h,u_h) -k^2 \norml{u_h}{\Om}^2 + \i k \norml{u_h}{\Ga_R}^2
=(f, u_h)+\pd{g, u_h}_{\Ga_R}.
\end{equation}
Therefore, taking real part and imaginary part of the above equation
and using \eqref{ea1} and \eqref{ea2} we get the following lemma.

\begin{lemma}\label{lem3.1}
Let $u_h\in V^h$ solve \eqref{edg}. Then
\begin{align}\label{e3.1}
&\abs{u_h}_{1,h}^2-2\re \sum_{e\in\cE_h^{ID}}\pdaj{\frac{\pa u_h}{\pa n_e}}{u_h}
-k^2 \norml{u_h}{\Om}^2 \leq \abs{(f,u_h)+\pd{g, u_h}_{\Ga_R}},\\
&\sum_{e\in\cE_h^{ID}}\left(\frac{\ga_{0,e}}{h_e} \norml{\jm{u_h}}{e}^2
+\sum_{j=1}^{d-1}\frac{\be_{1,e}}{h_e}\norml{\jm{\frac{\pa u_h}{\pa \tau_e^j}}}{e}^2\right) +\sum_{e\in\cE_h^I}\ga_{1,e} h_e\norml{\jm{\frac{\pa u_h}{\pa n_e}}}{e}^2\label{e3.2} \\
&\hskip 2.7in
+k\norml{u_h}{\Ga_R}^2\leq \bigl|(f,u_h)+\pd{g, u_h}_{\Ga_R}\bigr|. \nn
\end{align}
\end{lemma}

From \eqref{e3.1} and \eqref{e3.2} we can easily bound $\abs{u_h}_{1,h}$ and
the jumps in terms of $\norml{u_h}{\Om}$. In order to get the desired a priori
estimates, we now need to derive a reverse inequality whose coefficient
on the right-hand side can be controlled.  Such a reverse inequality,
which is often difficult to get under practical mesh constraints,
and stability estimates for scheme \eqref{edg} will be derived next.

\begin{theorem}\label{thm_sta}
Let $u_h\in V^h$ solve \eqref{edg} and suppose $\ga_{0,e}, \ga_{1,e}, \be_{1,e}>0$.  Define $M(f,g) :=\norml{f}{\Om}+\|g\|_{L^2(\Ga_R)}$. Then there exists a positive constant $\csta$ such that
\begin{align}\label{e3.3}
\norml{u_h}{\Om} &+\frac{1}{k} \norm{u_h}_{1,h}
+\frac{1}{k} \Bigl( \sum_{e\in\cE_h^R}\norml{\na u_h}{e}^2 \Bigr)^{\frac12} \\
&+ \frac{1}{k} \Bigl(\sum_{e\in\cE_h^D} c_D\bigl( k^2\norml{u_h}{e}^2
+\norml{\na u_h}{e}^2 \bigr) \Bigr)^{\frac12}
\lesssim \csta\, M(f,g) \nonumber
\end{align}
and
\begin{align}
\csta &\le\frac1k+\frac{1}{k^2}+\frac{1}{k^2}\max_{e\in\cE_h^D}\left(\frac{1}{\ga_{0,e}}+\frac{\ga_{0,e}}{h_e}+\sqrt{\frac{\ga_{0,e}}{\be_{1,e}}}+\frac{1}{\be_{1,e}}
\right)\label{csta}\\ & \qquad\qquad\;\;+\frac{1}{k^2}\max_{e\in\cE_h^I}\left(\frac{k^2+1}{\ga_{0,e}}+\frac{1}{h_e}\sqrt{\frac{\ga_{0,e}}{\ga_{1,e}}}
+\sqrt{\frac{\ga_{0,e}}{\be_{1,e}}}+\frac{1}{\be_{1,e}}\right).\nn
\end{align}
\end{theorem}

\begin{proof}
Since the proof is long, we divide it into three steps.

{\em Step 1: A representation identity for $\norml{u_h}{\Om}$}.
It follows from \eqref{eid1} with $v=u_h$ that
\[
d \norml{u_h}{K}^2 =  \int_{\pa K}\al\cdot n_K\abs{u_h}^2
-2\re (u_h,\al \cdot \nab u_h)_K.
\]
Summing over all $K\in\cT_h$ yields
\[
d \norml{u_h}{\Ome}^2 =  \sum_{K\in\cT_h} \int_{\pa K}\al\cdot n_K\abs{u_h}^2
-2\sum_{K\in\cT_h}\re (u_h,\al \cdot \nab u_h)_K,
\]
hence,
\begin{equation}\label{added1}
2k^2 \norml{u_h}{\Ome}^2 =  k^2\sum_{K\in\cT_h} \int_{\pa K}\al\cdot n_K\abs{u_h}^2
-(d-2) k^2 \norml{u_h}{\Ome}^2 -2k^2\re (u_h, v_h),
\end{equation}
where $v_h\in E$ is defined by $v_h|_K=\al\cdot\na u_h|_K$ for every $K\in \cT_h$.
It is easy to check that $v_h|_K$ is a linear polynomial on $K$, 
hence, $v_h\in V^h$.  Using this $v_h$ as a test function in \eqref{edg} 
and taking the real part of the resulted equation we get
\begin{equation}\label{added2}
-k^2 \re (u_h, v_h) = \re \bigl( (f,v_h)+\pd{g, v_h}_{\Ga_R}
-a_h(u_h,v_h)-\i k\pd{u_h,v_h}_{\Ga_R} \bigr).
\end{equation}

Now,  it follows from \eqref{added1}, \eqref{euh} and \eqref{added2} that
\begin{align}
2k^2\norml{u_h}{\Om}^2 &=
k^2\sum_{K\in\cT_h}\int_{\pa K}\al\cdot n_K\abs{u_h}^2
+(d-2)\re \bigl( (f,u_h) + \pd{g, u_h}_{\Ga_R}-a_h(u_h,u_h) \bigr)  \nn\\
&\quad  +2\re \bigl( (f,v_h)+ \pd{g, v_h}_{\Ga_R}
-a_h(u_h,v_h)-\i k\pd{u_h,v_h}_{\Ga_R} \bigr)  \nn\db\\
&= k^2\sum_{K\in\cT_h}\int_{\pa K}\al\cdot n_K\abs{u_h}^2
+(d-2)\re \bigl( (f,u_h) + \pd{g, u_h}_{\Ga_R} \bigr) \nn \\
&\quad +2\re \bigl( (f,v_h) + \pd{g, v_h}_{\Ga_R} \bigr)
+ 2k\im\pd{u_h,v_h}_{\Ga_R} \label{e3.5}\\
&\quad -\sum_{K\in\cT_h}\left((d-2)\norml{\na u_h}{K}^2
+2\re(\na u_h,\na v_h)_K\right)\nn\\
&\quad +2\sum_{e\in\cE_h^{ID}}\left((d-2)\re\pdaj{\frac{\pa u_h}{\pa n_e}}{u_h}
+\re\pdaj{\frac{\pa u_h}{\pa n_e}}{v_h}\right.\nn\\
&\quad \left.+\re\pdja{u_h}{\frac{\pa v_h}{\pa n_e}} \right)
+2\im\big(J_0(u_h,v_h)+J_1(u_h,v_h) +L_1(u_h,v_h)\big). \nn
\end{align}
Using the identity $\abs{a}^2-\abs{b}^2=\re (a+b)(\bar a-\bar b)$ we have
\begin{align}\label{e3.6}
\sum_{K\in\cT_h}\int_{\pa K}\al\cdot n_K\abs{u_h}^2
=2\sum_{e\in\cE_h^I}\re \Langle \al\cdot n_e \av{u_h},\jm{u_h} \Rangle_e
+\Langle \al\cdot n_\Ome, |u_h|^2 \Rangle_{\p \Ome}.
\end{align}
From the Rellich identity \eqref{eid2} and noting that $\pd{\al\cdot n_e\av{\na u_h},\jm{\na u_h}}_e=\pd{\al\cdot n_e, |\na u_h|^2}_e$ for $e\in\cE_h^{D}$ we get
\begin{align}\label{e3.7}
&\sum_{K\in\cT_h}\Bigl((d-2)\norml{\na u_h}{K}^2+2\re(\na u_h,\na v_h)_K\Bigr)
=\sum_{K\in\cT_h}\int_{\pa K}\al\cdot n_K\abs{\na u_h}^2  \\
=&2\sum_{e\in\cE_h^I} \re  \pd{\al\cdot n_e\av{\na u_h},\jm{\na u_h} }_e
+\sum_{e\in\cE_h^{RD}} \pd{\al\cdot n_e, |\na u_h|^2}_e. \nn\\
=&2\sum_{e\in\cE_h^{ID}} \re  \pd{\al\cdot n_e\av{\na u_h},\jm{\na u_h}}_e
+\sum_{e\in\cE_h^{R}} \pd{\al\cdot n_e, |\na u_h|^2}_e-\sum_{e\in\cE_h^{D}} \pd{\al\cdot n_e, |\na u_h|^2}_e.\nn
\end{align}
On noting that $u_h$ is piecewise linear and $v_h|_K=(x-x_{\Om_1})\cdot\na u_h|_K$,
then $\na v_h|_K=\na u_h|_K$. Hence,
\begin{equation}\label{e3.7a}
\im\big(J_1(u_h,v_h)+L_1(u_h,v_h)\big)=\im\big(J_1(u_h,u_h)+L_1(u_h,u_h)\big)=0.
\end{equation}
Plugging \eqref{e3.6}--\eqref{e3.7a} into \eqref{e3.5}
then gives the following representation for $\norml{u_h}{\Om}$:
\begin{align}\label{e3.8}
2k^2\norml{u_h}{\Om}^2 &=(d-2)\re\bigl( (f,u_h) +\pd{g, u_h}_{\Ga_R} \bigr)
+2\re \bigl( (f,v_h)  +\pd{g, v_h}_{\Ga_R} \bigr) \nn\\
&\quad +2k^2\sum_{e\in\cE_h^I}\re\pd{\al\cdot n_e\av{u_h},\jm{u_h}}_e
+k^2\pd{\al\cdot n_\Ome, |u_h|^2}_{\pa\Om}\nn\\
&\quad +2k\im\pd{u_h,v_h}_{\Ga_R}-\sum_{e\in\cE_h^{R}}\pd{\al\cdot n_e, |\na u_h|^2}_e+\sum_{e\in\cE_h^{D}}\pd{\al\cdot n_e, |\na u_h|^2}_e\nn\\
&\quad +2\sum_{e\in\cE_h^{ID}}\re\left(-\pd{\al\cdot n_e\av{\na u_h},\jm{\na u_h}}_e
+\pdaj{\frac{\pa u_h}{\pa n_e}}{v_h}\right) \\
&\quad
+2\sum_{e\in\cE_h^{ID}}\left((d-1)\re\pdaj{\frac{\pa u_h}{\pa n_e}}{u_h}+\re\pdja{u_h}{\frac{\pa v_h}{\pa n_e}} \right)\nn\\
&\quad -2\sum_{e\in\cE_h^{ID}}\re\pdaj{\frac{\pa u_h}{\pa n_e}}{u_h}
+2\im J_0(u_h,v_h).  \nn
\end{align}

{\em Step 2: Derivation of a reverse inequality.}
We bound each terms on the right hand side of \eqref{e3.8}.
For an edge/face $e\in\cE_h^I$, let $K_e$ and $K_e^\pr$ denote the two elements
in $\cT_h$ that share $e$. For an edge/face $e\in\cE_h^D$, let $K_e$ denote the element
in $\cT_h$ that has $e$ as an edge/face and $K_e^\pr=\emptyset$. We have
\begin{align}\label{e3.9}
2k^2\sum_{e\in\cE_h^I}\re\pd{\al\cdot n_e\av{u_h},\jm{u_h}}_e
&\le Ck^2\sum_{e\in\cE_h^I}h_e^{-\frac12}\norml{u_h}{K_e\cup K_e^\pr}\norml{\jm{u_h}}{e} \\
&\le \frac{k^2}{3}\norml{u_h}{\Om}^2+C\sum_{e\in\cE_h^I}\frac{k^2}{\ga_{0,e}}\frac{\ga_{0,e}}{h_e}\norml{\jm{u_h}}{e}^2. \nn
\end{align}
It is clear that
\begin{align}\label{e3.10}
k^2\pd{\al\cdot n_\Ome, |u_h|^2}_{\p \Ome} &= k^2\pd{\al\cdot n_\Om, |u_h|^2}_{\Ga_R}
+\sum_{e\in\cE_h^{D}}k^2\pd{\al\cdot n_e, |u_h|^2}_e \\
&\le C k^2 \norml{u_h}{\Ga_R}^2 +\sum_{e\in\cE_h^{D}}k^2\pd{\al\cdot n_e, |u_h|^2}_e. \nn
\end{align}
It follows from the star-shaped assumption on $\Om_1$ that
\begin{align}\label{e3.11}
2k&\im\pd{u_h,v_h}_{\Ga_R}-\sum_{e\in\cE_h^{R}}\pd{\al\cdot n_e, |\na u_h|^2}_e \\
&\le Ck\sum_{e\in\cE_h^{R}}\norml{u_h}{e}\norml{\na u_h}{e}
-c_{\Ome_1}\sum_{e\in\cE_h^{R}}\norml{\na u_h}{e}^2  \nn\\
&\le Ck^2\norml{u_h}{\Ga_R}^2 -\frac{c_{\Ome_1}}2 \sum_{e\in\cE_h^{R}}\norml{\na u_h}{e}^2 . \nn
\end{align}
By \eqref{eid3} we obtain
\begin{align}\label{e3.12}
&2\sum_{e\in\cE_h^{ID}}\re\left(-\pd{\al\cdot n_e\av{\na u_h},\jm{\na u_h}}_e
+\pdaj{\frac{\pa u_h}{\pa n_e}}{v_h}\right) \\
&\hskip 0.5in
=2\sum_{e\in\cE_h^{ID}}\sum_{j=1}^{d-1} \re
\int_e\left(\al\cdot\tau_e^j\av{\frac{\pa u_h}{\pa n_e}}
-\al\cdot n_e\av{\frac{\pa u_h}{\pa \tau_e^j}}\right)
\frac{\pa\jm{\overline{u}_h}}{\pa\tau_e^j} \nn\\
&\hskip 0.5in
\lesssim\sum_{e\in\cE_h^{ID}}\sum_{j=1}^{d-1}h_e^{-\frac12}\sum_{K=K_e, K_e^\pr}
\norml{\na u_h}{K}
\norml{\jm{\frac{\pa u_h}{\pa\tau_e^j}}}{e}  \nn\\
&\hskip 0.5in
\le \frac13\abs{u_h}_{1,h}^2
+C\sum_{e\in\cE_h^{ID}}\frac{1}{\be_{1,e}}\sum_{j=1}^{d-1}\frac{\be_{1,e}}{h_e}\norml{\jm{\frac{\pa u_h}{\pa\tau_e^j}}}{e}^2. \nn
\end{align}
Noting that $\frac{\pa v_h}{\pa n_e}=\frac{\pa u_h}{\pa n_e}$ we have
\begin{align}\label{e3.13}
&2\sum_{e\in\cE_h^{ID}}\left((d-1)\re\pdaj{\frac{\pa u_h}{\pa n_e}}{u_h}
+\re\pdja{u_h}{\frac{\pa v_h}{\pa n_e}}\right) \\
&\hskip 0.5in
\lesssim\sum_{e\in\cE_h^{ID}}h_e^{-\frac12}\sum_{K=K_e, K_e^\pr}
\norml{\na u_h}{K}\norml{\jm{u_h}}{e} \nn\\
&\hskip 0.5in
\le \frac13\abs{u_h}_{1,h}^2+C\sum_{e\in\cE_h^{ID}}\frac{1}{\ga_{0,e}}\frac{\ga_{0,e}}{h_e}\norml{\jm{u_h}}{e}^2. \nn
\end{align}
From \eqref{eJ0}, the inverse inequality and \eqref{e3.2} we get
\begin{align} \label{e3.14}
2\im\big(&J_0(u_h,v_h)\big) =2\im\sum_{e\in\cE_h^{ID}}\frac{\ga_{0,e}}{h_e} \pdjj{u_h}{v_h}  \\
&=2\im\sum_{e\in\cE_h^{ID}}\frac{\ga_{0,e}}{h_e}\pdjj{u_h}{\al\cdot n_e\frac{\pa u_h}{\pa n_e}
+\sum_{j=1}^{d-1}\al\cdot\tau_e^j\frac{\pa u_h}{\pa \tau_e^j}} \nn\db \\
&\le 2\im\sum_{e\in\cE_h^D}\frac{\ga_{0,e}}{h_e}\pd{\al\cdot n_e u_h,\frac{\pa u_h}{\pa n_e}}_e+C\sum_{e\in\cE_h^I}\frac{\ga_{0,e}}{h_e} \norml{\jm{u_h}}{e}
\norml{\jm{\frac{\pa u_h}{\pa n_e}}}{e}
\nn\\
&\quad+C\sum_{e\in\cE_h^{ID}}\frac{\ga_{0,e}}{h_e}\norml{\jm{u_h}}{e}\sum_{j=1}^{d-1}\norml{\jm{\frac{\pa u_h}{\pa\tau_e^j}}}{e} \nn\db\\
&\le 2\im\sum_{e\in\cE_h^D}\frac{\ga_{0,e}}{h_e}\pd{\al\cdot n_e u_h,\frac{\pa u_h}{\pa n_e}}_e
\nn\\
 &\quad+C\sum_{e\in\cE_h^I}\sqrt{\frac{\ga_{0,e}}{\ga_{1,e}}}\frac{1}{h_e} \left(\frac{\ga_{0,e}}{h_e}\norml{\jm{u_h}}{e}^2+
\ga_{1,e}h_e\norml{\Jump{\frac{\pa u_h}{\pa n_e}}}{e}^2\right)\nn\\
&\quad+C\sum_{e\in\cE_h^{ID}}\sqrt{\frac{\ga_{0,e}}{\be_{1,e}}}\left(\frac{\ga_{0,e}}{h_e}\norml{\jm{u_h}}{e}^2+
\sum_{j=1}^{d-1}\frac{\be_{1,e}}{h_e}\norml{\jm{\frac{\pa u_h}{\pa\tau_e^j}}}{e}^2\right). \nn
\end{align}
Since $D$ is star-shaped, we have
\begin{align}\label{e3.15}
&\sum_{e\in\cE_h^{D}} \left( k^2\pd{\al\cdot n_e, |u_h|^2}_e
+\pd{\al\cdot n_e, |\na u_h|^2}_e
+\frac{2\ga_{0,e}}{h_e}\im\pd{\al\cdot n_e u_h,\frac{\pa u_h}{\pa n_e}}_e \right) \\
&=-\sum_{e\in\cE_h^{D}}  \left( k^2\pd{\al\cdot n_D, |u_h|^2}_e
+\pd{\al\cdot n_D, |\na u_h|^2}_e
+\frac{2\ga_{0,e}}{h_e} \im\pd{\al\cdot n_D u_h,\frac{\pa u_h}{\pa n_D}}_e \right) \nonumber \db\\
&\le -\sum_{e\in\cE_h^{D}}\pd{\al\cdot n_D, k^2|u_h|^2+|\na u_h|^2-2\frac{\ga_{0,e}}{h_e}|u_h||\na u_h|}_e \nonumber \db\\
&\le -\sum_{e\in\cE_h^{D}}\pd{\al\cdot n_D, k^2|u_h|^2+\frac12|\na u_h|^2-2\frac{\ga_{0,e}}{h_e}\frac{\ga_{0,e}}{h_e}|u_h|^2}_e \nonumber \\
&\le -c_D\sum_{e\in\cE_h^{D}} \Bigl( k^2\|u_h\|_{L^2(e)}^2 + \frac12 \|\na u_h\|_{L^2(e)}^2\Bigr)
+C\sum_{e\in\cE_h^{D}}\frac{\ga_{0,e}}{h_e} \frac{\ga_{0,e}}{h_e}\norml{u_h}{e}^2 , \nonumber
\end{align}

Putting \eqref{e3.8}--\eqref{e3.15} together we have
\begin{align*}
2k^2\norml{u_h}{\Om}^2 &\le (d-2)\re\bigl( (f,u_h) +\pd{g, u_h}_{\Ga_R} \bigr)
+2\re \bigl( (f,v_h)  +\pd{g, v_h}_{\Ga_R} \bigr) + \frac{k^2}{3}\norml{u_h}{\Om}^2\\
&\quad +C\sum_{e\in\cE_h^I}\frac{k^2}{\ga_{0,e}}\frac{\ga_{0,e}}{h_e}\norml{\jm{u_h}}{e}^2+ C k^2 \norml{u_h}{\Ga_R}^2-\frac{c_{\Ome_1}}2 \sum_{e\in\cE_h^R}\norml{\na u_h}{e}^2  \\
&\quad  +\frac23\abs{u_h}_{1,h}^2
+C\sum_{e\in\cE_h^{ID}}\left(\frac{1}{\be_{1,e}}\sum_{j=1}^{d-1}\frac{\be_{1,e}}{h_e}\norml{\jm{\frac{\pa u_h}{\pa\tau_e^j}}}{e}^2+\frac{1}{\ga_{0,e}}\frac{\ga_{0,e}}{h_e}\norml{\jm{u_h}}{e}^2\right) \\
&\quad
-2\sum_{e\in\cE_h^{ID}}\re\pdaj{\frac{\pa u_h}{\pa n_e}}{u_h}-c_D\sum_{e\in\cE_h^{D}} \Bigl( k^2\|u_h\|_{L^2(e)}^2 + \frac12 \|\na u_h\|_{L^2(e)}^2\Bigr)\nn\\
&\quad+C\sum_{e\in\cE_h^{D}}\frac{\ga_{0,e}}{h_e}\frac{\ga_{0,e}}{h_e}\norml{u_h}{e}^2\\
&\quad+C\sum_{e\in\cE_h^I}\sqrt{\frac{\ga_{0,e}}{\ga_{1,e}}}\frac{1}{h_e} \left(\frac{\ga_{0,e}}{h_e}\norml{\jm{u_h}}{e}^2+
\ga_{1,e}h_e\norml{\Jump{\frac{\pa u_h}{\pa n_e}}}{e}^2\right)\nn
\end{align*}
\begin{align*}
&\quad+C\sum_{e\in\cE_h^{ID}}\sqrt{\frac{\ga_{0,e}}{\be_{1,e}}}\left(\frac{\ga_{0,e}}{h_e}\norml{\jm{u_h}}{e}^2+
\sum_{j=1}^{d-1}\frac{\be_{1,e}}{h_e}\norml{\jm{\frac{\pa u_h}{\pa\tau_e^j}}}{e}^2\right). \nn\\
\end{align*}
Therefore from \eqref{e3.2},
\begin{align*}
&2k^2\norml{u_h}{\Om}^2
+\frac{c_{\Ome_1}}2 \sum_{e\in\cE_h^R}\norml{\na u_h}{e}^2
+ c_D\sum_{e\in\cE_h^{D}} \Bigl( k^2\|u_h\|_{L^2(e)}^2 + \frac12 \|\na u_h\|_{L^2(e)}^2\Bigr) \db\\
&\quad \le \frac{k^2}{3}\norml{u_h}{\Om}^2+\frac23\abs{u_h}_{1,h}^2
-2\sum_{e\in\cE_h^{ID}}\re\pdaj{\frac{\pa u_h}{\pa n_e}}{u_h} +2\re \bigl( (f,v_h)  +\pd{g, v_h}_{\Ga_R} \bigr)\db \\
&\qquad +C\left(k+1+\max_{e\in\cE_h^D}\left(\frac{1}{\ga_{0,e}}+\frac{\ga_{0,e}}{h_e}+\sqrt{\frac{\ga_{0,e}}{\be_{1,e}}}+\frac{1}{\be_{1,e}}\right)\right.\db\\
&\qquad\qquad\qquad\quad  \left.+\max_{e\in\cE_h^I}\left(\frac{k^2+1}{\ga_{0,e}}+\frac{1}{h_e}\sqrt{\frac{\ga_{0,e}}{\ga_{1,e}}}+\sqrt{\frac{\ga_{0,e}}{\be_{1,e}}}+\frac{1}{\be_{1,e}}\right)
\right)\abs{(f,u_h)+\pd{g, u_h}_{\Ga_R}}.
\end{align*}

\medskip
{\em Step 3: Finishing up.} It follows from \eqref{e3.1}, \eqref{e3.2}, \eqref{csta}, and the above inequality that
\begin{align*}
2&k^2\norml{u_h}{\Om}^2+\frac{c_{\Ome_1}}2 \sum_{e\in\cE_h^R}\norml{\na u_h}{e}^2
+ c_D\sum_{e\in\cE_h^{D}} \Bigl( k^2\|u_h\|_{L^2(e)}^2 + \frac12 \|\na u_h\|_{L^2(e)}^2\Bigr) \\
& \le Ck^2\csta\abs{(f,u_h)+\pd{g, u_h}_{\Ga_R} }  +\frac{4k^2}{3}\norml{u_h}{\Om}^2 -\frac13\abs{u_h}_{1,h}^2 +2\re \bigl( (f,v_h)  +\pd{g, v_h}_{\Ga_R} \bigr)\\
& \le \frac{5k^2}{3} \norml{u_h}{\Om}^2-\frac16\abs{u_h}_{1,h}^2+\frac{c_{\Ome_1}}4 \sum_{e\in\cE_h^R}\norml{\na u_h}{e}^2+Ck^2\csta^2  M(f,g)^2,
\end{align*}
where $M(f,g) :=\norml{f}{\Om}+\|g\|_{L^2(\Ga_R)}$ and we have used $k^2\norml{u_h}{\Ga_R}^2\le k^2\norml{u_h}{\Om}^2+M(f,g)^2$ (cf. \eqref{e3.2}) to derive the last inequality.  Hence,
\begin{align*}
\norml{u_h}{\Om} + \frac{1}{k} \abs{u_h}_{1,h}
&+ \frac{1}{k} \Bigl( \sum_{e\in\cE_h^R}\norml{\na u_h}{e}^2 \Bigr)^{\frac12} \\
&+ \frac{1}{k} \Bigl( \sum_{e\in\cE_h^{D}} c_D\bigl( k^2\|u_h\|_{L^2(e)}^2
+ \frac12 \|\na u_h\|_{L^2(e)}^2\bigr)  \Bigr)^{\frac12}
\lesssim \csta M(f,g),
\end{align*}
which together with \eqref{e3.2} implies  \eqref{e3.3}. The proof is completed.
\end{proof}

\begin{remark}
If the penalty parameters are taken as $\ga_{1,e}\equiv 0$, $\be_{1,e}\equiv 0$ and $\ga_{0,e}>0$,
the terms $\norml{\jm{\frac{\pa u_h}{\pa\tau_e^j}}}{e}=\norml{\frac{\pa \jm{u_h}}{\pa\tau_e^j}}{e}$ in \eqref{e3.12} and \eqref{e3.14} ought be estimated differently
by using the inverse inequality. This then leads to the following weaker
stability estimate:
\begin{align}\label{e3.3a}
&\norml{u_h}{\Om} + \frac{1}{k} \|u_h\|_{1,h}\\
&\quad\lesssim\max_{e\in\cE_h^{ID}}\left(\frac1{k}+\frac1{k^2}+\frac{k^2+1}{\ga_{0,e}k^2}
+\frac{1+\ga_{0,e}^2(1+h_e)+\ga_{0,e}h_e}{\ga_{0,e}(kh_e)^2}\right) M(f,g).\nn
\end{align}
\end{remark}

Since scheme \eqref{edg} is a linear complex-valued system,  an immediate consequence
of the stability estimates is the following well-posedness theorem for \eqref{edg}.

\begin{theorem}\label{existence}
The IPDG method \eqref{edg} has a unique solution for $k>0, h_e>0, \ga_{0,e}>0$,
$\sigma= 1$, $\ga_{1,e}\ge 0$ and $\be_{1,e}\ge 0$.
\end{theorem}

\begin{remark}
(a) IPDG method \eqref{edg} is well-posed for all wave number $k>0$
provided that the penalty parameter $\ga_{0,e}>0$. As a comparison, we recall
that \cite{ib95a} the finite element method is well-posed only if mesh size $h$
satisfies a constraint $h=O(k^{-r})$ for some $r\geq 1$, hence, the existence
is only guaranteed for very small mesh size $h$ when wave number $k$ is large.

(b). It is well known that \cite{arnold82,abcm01,baker77,fk07,rwg99,w78}
symmetric IPDG methods for coercive elliptic and parabolic PDEs
often require the penalty parameter $\ga_{0,e}$ is sufficiently
large to guarantee the well-posedness of numerical solutions, and the
low bound for $\ga_{0,e}$ is hard to determine and is also problem-dependent.
However, this is no issue for scheme \eqref{edg}, which solves
the (indefinite) Helmholtz equation, because it is well-posed for
all $\ga_{0,e}>0$.
\end{remark}

We have the following consequence of Theorem~\ref{thm_sta} for quasi-uniform meshes.
\begin{theorem}\label{thm_sta_uniform}Let $h=\max h_e.$ Suppose the mesh $\cT_h$ is quasi-uniform, that is $h_e\simeq h.$ Assume that $\ga_{1,e}\simeq\ga_1>0$,
and that $\ga_{0,e}\simeq (k^2h)^{2/3}\ga_1^{1/3}$
and $\be_{1,e} \gtrsim (h/k)^{2/3}\ga_1^{1/3}$ for $e\in\cE_h^I$, and that $\ga_{0,e}\simeq \ga_0^D>0$
and $\be_{1,e} \gtrsim \ga_0^D$ for $e\in\cE_h^D$, where $\ga_1$ and $\ga_0^D$ are independent of $e$. If $k\gtrsim 1$, then
\begin{equation}\label{esta}
\norml{u_h}{\Om}+\frac1k\norm{u_h}_{1,h}\lesssim
\left(\frac1k+\frac{1}{k^2}\left(\frac{1}{\ga_0^D}
+\frac{\ga_0^D}{h}\right)+\frac{1}{(k^2h)^{2/3}\ga_1^{1/3}}\right) M(f,g).
\end{equation}
If, furthermore, $\ga_0^D\simeq 1$ and $\ga_1\lesssim k^2h$, then
\begin{equation}\label{esta1}
\norml{u_h}{\Om}+\frac1k\norm{u_h}_{1,h}\lesssim
\left(\frac1k+\frac{1}{(k^2h)^{2/3}\ga_1^{1/3}}\right) M(f,g).
\end{equation}
\end{theorem}
\begin{remark}
It is clear that if $k^2h\gtrsim 1$ and $\ga_1$ and $\ga_0^D$ are chosen properly, say $\ga_1\gtrsim \frac1{kh^2}$, $\ga_0^D\simeq \sqrt{h}$, then
\[\norml{u_h}{\Om}+\frac1k\norm{u_h}_{1,h}\lesssim \frac1k M(f,g).
\]
Note that the above estimate is of the same order as the PDE stability estimate given
in Theorem \ref{stability}.  But a large $\ga_1$ or a small $\ga_0^D$ may cause a
large error (cf. Theorem~\ref{thm main 2}).
\end{remark}

\section{Error analysis}\label{sec-err}
In this section, we derive the error estimates for the solution
of scheme \eqref{edg}. This will be done in two steps.
First, we introduce an elliptic projection of the PDE solution
$u$ and derive error estimates for the projection.
We note that such a result also has an independent interest.
Second, we bound the error between the projection and the IPDG
solution by making use of the stability results obtained in
Section~\ref{sec-sta}.  In this section, we assume that the
mesh $\cT_h$ is quasi-uniform and $\ga_{1,e}\simeq\ga_1> 0$.
Also, we define
\[
\ga_0:=\min_{e\in\cE_h^{ID}}\ga_{0,e}\, (>0).
\]

\subsection{Elliptic projection and its error estimates}
For any $w\in E\cap H_{\Ga_D}^1(\Om)\cap H^2_{\loc} (\Ome)$, we define its
elliptic projection $\tilde{w}_h\in V^h$ by
\begin{equation}\label{e4.3}
a_h(\tilde{w}_h,v_h)+\i k\pd{\tilde{w}_h,v_h}_{\Ga_R} =a_h(w,v_h)+\i k\pd{w,v_h}_{\Ga_R}
\qquad\forall v_h\in V^h.
\end{equation}
In other words, $\tilde{w}_h$ is an IPDG approximation to
the solution $w$ of the following (complex-valued) Poisson problem:
\begin{alignat*}{2}
-\Del w &= F &&\qquad \mbox{in } \Ome,\\
w &= 0 &&\qquad \mbox{on }\Ga_D,\\
\frac{\p w}{\p n_\Om} +\i k w &=\psi &&\qquad \mbox{on }\Ga_R,
\end{alignat*}
for some given functions $F$ and  $\psi$ which are determined by $w$.

Before estimating the projection error, we state the following continuity and
coercivity properties for the sesquilinear form $a_h(\cdot,\cdot)$.
Since they follow easily from \eqref{eah}--\eqref{ea2}, so
we omit their proofs to save space.
\begin{lemma}\label{lem4.1}
For any $v\in E$ and $w\in E\cap H_{\Ga_D}^1(\Om)$, the mesh-dependent
sesquilinear form $a_h(\cdot,\cdot)$ satisfies
\begin{equation}
\abs{a_h(v,w)}, \;\abs{a_h(w,v)}\lesssim \norm{v}_{1,h} \norme{w}. \label{e4.1}
\end{equation}
In addition, for any $0<\ep<1$, there exists a positive constant
$c_\ep$ such that
\begin{equation}
\re a_h(v_h,v_h)+\Bigl(1-\ep+\frac{c_\ep}{\ga_0}\Bigr)
\im a_h(v_h,v_h)\ge (1-\ep)\norm{v_h}_{1,h}^2\quad \forall v_h\in V^h. \label{e4.2}
\end{equation}
\end{lemma}

Let $u$ be the solution of problem \eqref{e1.1}--\eqref{e1.3}
and $\tuh$ be its elliptic projection defined as above. Then \eqref{e4.3} immediately implies the following
Galerkin orthogonality:
\begin{equation}\label{e4.4}
a_h(u-\tuh,v_h)+\i k\pd{u-\tuh,v_h}_{\Ga_R} =0 \qquad\forall v_h\in V^h.
\end{equation}

\begin{lemma}\label{lem-4.2}
Suppose problem \eqref{e1.1}--\eqref{e1.3} is $H^2$-regular. Then there hold the
following estimates:
\begin{align}
\norm{u-\tuh}_{1,h} + (\la k)^{\frac12} \norml{u-\tuh}{\Ga_R}
&\lesssim \la\big(\la+\ga_1+kh\big)^{\frac12} kh, \label{e4.2a}\\
\norml{u-\tuh}{\Om} &\lesssim \la\big(\la+\ga_1+kh\big) kh^2,  \label{e4.2b}
\end{align}
where $\la:=1+\frac{1}{\gamma_0}$.
\end{lemma}

\begin{proof}
Let $\hat{u}_h$ be the $P_1$-conforming finite element interpolation of $u$ on
the mesh $\cT_h$.  Then $\hat{u}_h\in E\cap H_{\Ga_D}^1(\Om)$ and satisfies the following
estimates (cf. \cite{bs94,ciarlet78}):
\begin{align}\label{e4.6}
\norml{u-\hat{u}_h}{\Om} &\lesssim h^2 \abs{u}_{H^2(\Om)}, \\
\norme{u-\hat{u}_h} &\lesssim \Big(1+\ga_1+\frac{1}{\ga_0}\Big)^{\frac12} \,h\abs{u}_{H^2(\Om)}=\big(\la+\ga_1\big)^{\frac12} \,h\abs{u}_{H^2(\Om)}, \label{e4.6b}\\
\norml{u-\hat{u}_h}{\Ga_R} &\lesssim h^{\frac32} \abs{u}_{H^2(\Om)}.
\label{e4.6c}
\end{align}
Let $\eta_h:=\tuh-\hat{u}_h$. From $\eta_h+u-\tuh=u-\hat{u}_h$ and \eqref{e4.4},
\begin{equation}\label{e4.5}
a_h(\eta_h,\eta_h)+\i k\pd{\eta_h,\eta_h}_{\Ga_R}
=a_h(u-\hat{u}_h,\eta_h)+\i k\pd{u-\hat{u}_h,\eta_h}_{\Ga_R}.
\end{equation}
Take  $\ep=\frac12$ in \eqref{e4.2} and assume without loss of generality
that $c_{\frac12}>\frac12$. It follows from \eqref{e4.2} and \eqref{e4.5} that
\begin{align*}
\frac12\norm{\eta_h}_{1,h}^2\le& \re a_h(\eta_h,\eta_h)
+\Bigl(\frac12+\frac{c_{\frac12}}{\ga_0}\Bigr)\im a_h(\eta_h,\eta_h)\\
=&\re\bigl(a_h(u-\hat{u}_h,\eta_h)+\i k\pd{u-\hat{u}_h,\eta_h}_{\Ga_R} \bigr)
-\Bigl(\frac12+\frac{c_{\frac12}}{\ga_0}\Bigr)k\pd{\eta_h,\eta_h}_{\Ga_R} \\
&\qquad +\Bigl(\frac12+\frac{c_{\frac12}}{\ga_0}\Bigr)
\im\left(a_h(u-\hat{u}_h,\eta_h)+\i k\pd{u-\hat{u}_h,\eta_h}_{\Ga_R} \right)\\
\le &C\la \Bigl(\norm{\eta_h}_{1,h}\norme{u-\hat{u}_h}
+k\norml{u-\hat{u}_h}{\Ga_R}^2\Bigr)-\frac{\la k}{4}\norml{\eta_h}{\Ga_R}^2.
\end{align*}
Therefore, it follows from \eqref{e4.6b}, \eqref{e4.6c} and \eqref{stability} that
\begin{align}\label{e4.6a}
\norm{\eta_h}_{1,h}^2 +\la k\norml{\eta_h}{\Ga_R}^2
\lesssim &\la^2\norme{u-\hat{u}_h}^2
+\la k\norml{u-\hat{u}_h}{\Ga_R}^2\nn\\
\lesssim&\la^2 k^2h^2\big(\la+\ga_1+kh\big).
\end{align}
which together with the fact that $u-\tuh=u-\hat{u}_h-\eta_h$ yields \eqref{e4.2a}.

To show \eqref{e4.2b}, we use the Nitsche's duality argument (cf. \cite{bs94,ciarlet78}).
Consider the following auxiliary problem:
\begin{alignat}{2}\label{e4.7}
-\De w &=u-\tuh &&\qquad\text{in }\Om,\\
w&=0 &&\qquad\text{on }\Ga_D, \nn \\
\frac{\pa w}{\pa n_\Om}-\i k w &=0 &&\qquad\text{on }\Ga_R. \nn
\end{alignat}
It can be shown that $w$ satisfies
\begin{equation}\label{e4.8}
\abs{w}_{H^2(\Om)}\lesssim \norml{u-\tuh}{\Om}.
\end{equation}
Let $\hat{w}_h$ be the $P_1$-conforming finite element interpolation of $w$ on
$\cT_h$. Testing the conjugated \eqref{e4.7} by $u-\tuh$ and using \eqref{e4.4} we get
\begin{align*}
\norml{u-\tuh}{\Om}^2 &=-(u-\tuh, \De w)=a_h(u-\tuh,w)+\i k\pd{u-\tuh,w}_{\Ga_R} \\
&=a_h(u-\tuh,w-\hat{w}_h)+\i k\pd{u-\tuh,w-\hat{w}_h}_{\Ga_R} \\
&\le \norm{u-\tuh}_{1,h}\norme{w-\hat{w}_h}
+k\norml{u-\tuh}{\Ga_R}\norml{w-\hat{w}_h}{\Ga_R} \\
&\lesssim \norm{u-\tuh}_{1,h} \big(\la+\ga_1\big)^{\frac12} h \abs{w}_{H^2(\Om)}
+k\norml{u-\tuh}{\Ga_R} h^{\frac32}\abs{w}_{H^2(\Om)}.
\end{align*}
which together with \eqref{e4.2a} and \eqref{e4.8} gives \eqref{e4.2b}.
The proof is completed.
\end{proof}

\subsection{Error estimates for scheme \eqref{edg}}
In this subsection we shall derive error estimates for scheme \eqref{edg}.
This will be done by exploiting the linearity of the Helmholtz equation
and making use of the stability estimates derived in Theorem \ref{thm_sta} and
the projection error estimates established in Lemma \ref{lem-4.2}.

Let $u$ and $u_h$ denote the solution of  \eqref{e1.1}--\eqref{e1.3} and that of \eqref{edg},
respectively. Assume that $u\in H^2(\Om).$ Then \eqref{2.6} holds for $v=v_h\in V^h.$ Define the error function $e_h:=u-u_h$. Subtracting \eqref{edg}
from \eqref{2.6} yields the following error equation:
\begin{equation}\label{error-eq}
a_h(e_h,v_h) -k^2 (e_h,v_h) +\i k \langle e_h,v_h\rangle_{\Ga_R} =0
\qquad \forall v_h\in V^h.
\end{equation}
Let $\tuh$ be the elliptic projection of $u$ as defined in the previous
subsection. Write $e_h=\eta-\xi$ with
\[
\qquad \eta:=u-\tuh, \qquad \xi:=u_h-\tuh.
\]
From \eqref{error-eq} and \eqref{e4.4} we get
\begin{align}\label{e4.14}
a_h(\xi,v_h) -k^2 (\xi,v_h) +\i k \langle \xi,v_h\rangle_{\Ga_R}
&= a_h(\eta,v_h) -k^2 (\eta,v_h) +\i k \langle \eta,v_h\rangle_{\Ga_R} \\
&=-k^2 (\eta,v_h) \qquad \forall v_h\in V^h. \nn
\end{align}
The above equation implies that $\xi\in V^h$ is the solution of
scheme \eqref{edg} with source terms $f=-k^2\eta$ and $g\equiv 0$.
Then an application of Theorem \ref{thm_sta} and Lemma \ref{lem-4.2}
immediately gives the following lemma.

\begin{lemma}\label{lem-4.3}
$\xi=u_h-\tuh$ satisfies the following estimate:
\begin{align}\label{est_xi}
&\norml{\xi}{\Om}+\frac{1}{k} \norm{\xi}_{1,h}
\lesssim \csta\la\big(\la+\ga_1+kh\big) k^3h^2,
\end{align}
where $\csta$ is defined in \eqref{csta} and $\la=1+\frac{1}{\gamma_0}.$
\end{lemma}

We are ready to state our error estimate results for scheme \eqref{edg},
which follows from Lemma \ref{lem-4.3},  Lemma \ref{lem-4.2} and an
application of the triangle inequality.

\begin{theorem}\label{thm_main}
Let $u$ and $u_h$ denote the solutions of \eqref{e1.1}--\eqref{e1.3}
and \eqref{edg}, respectively. Assume that $u\in H^2(\Om).$
Then there exist two positive constants $C_1$ and $C_2$ such that the following error estimates hold.
\begin{align}\label{e4.15}
\norm{u-u_h}_{1,h} &\le \la\big(\la+\ga_1+kh\big)\big(C_1 kh+C_2\csta k^4h^2\big),  \\
\norml{u-u_h}{\Ome} &\le  \la\big(\la+\ga_1+kh\big) \big(C_1 kh^2+C_2\csta k^3h^2\big),\label{e4.16}
\end{align}
where $\csta$ is defined in Theorem~\ref{thm_sta} and $\la=1+\frac{1}{\gamma_0}.$
\end{theorem}

From Theorem \ref{thm_sta_uniform} and the definition of $\csta$ 
(cf. \eqref{csta}) we obtain the following estimates.
\begin{theorem}\label{thm main 2} 
Assume that $0<\ga_1\lesssim k^2h$,
and that $\ga_{0,e}\simeq (k^2h)^{2/3}\ga_1^{1/3}$
and $\be_{1,e} \gtrsim (h/k)^{2/3}\ga_1^{1/3}$ for $e\in\cE_h^I$, and that $\ga_{0,e}\simeq 1$
and $\be_{1,e} \gtrsim 1$ for $e\in\cE_h^D$. If $k\gtrsim 1$, then there exist two positive constants $C_1$ and $C_2$ such that the following error estimates hold.
\begin{align}\label{e4.17}
\norm{u-u_h}_{1,h} &\le \la\big(\la+\ga_1+kh\big)\Big(C_1kh + C_2\Big(\frac1k+\frac{1}{(k^2h)^{2/3}\ga_1^{1/3}}\Big)k^4h^2\Big), \\
\norml{u-u_h}{\Ome} &\le  \la\big(\la+\ga_1+kh\big) \Big(C_1 kh^2+C_2\Big(\frac1k+\frac{1}{(k^2h)^{2/3}\ga_1^{1/3}}\Big) k^3h^2\Big). \label{e4.18}
\end{align}

\end{theorem}
\begin{remark}
(a) The estimates in \eqref{e4.15}--\eqref{e4.18} are so-called preasymptotic
error estimates (i.e. for the mesh in the regime $k^2h\gtrsim 1$). In fact,
the estimates hold for any $h>0$. We recall that \cite{ib95a} the preasymptotic
error estimates for the finite element method solution
was only proved in the $1$-d case provided that $kh\le 1$.

(b) The second term on the right-hand side of the first inequality is pollution
term for $\norm{u-u_h}_{1,h}$.

(c) If $kh\lesssim 1$, then under the assumption of Theorem~\ref{thm main 2} 
we have 
\begin{equation}\label{e4.19}
\norm{u-u_h}_{1,h} \le \widetilde{C}_1kh + \widetilde{C}_2k^{8/3}h^{4/3}
\end{equation}
for some constants $\widetilde{C}_1$ and $\widetilde{C}_2$ which depend on $\ga_1.$
Numerical tests in the next section suggest that $\norm{u-u_h}_{1,h}$
may have a better bound $\widetilde{C}_1kh + \widetilde{C}_2k^{3}h^{2}$ and
it is possible to tune the penalty parameters to significantly reduce the
pollution error.  We note that in the case $k^2h$ is sufficiently small, 
optimal order (with respect to $h$) error estimate in the 
broken $H^1$-norm can be derived by using the Schatz argument 
as done in \cite{perugia07, ak79, dssb93, dss94}.

(d) Inequality \eqref{est_xi} shows that $\norm{\tuh-u_h}_{1,h}$ enjoys a
superconvergence.
\end{remark}

\section{Numerical experiments}\label{sec-num}
Throughout this section, we consider the following two-dimensional
Helmholtz problem:
\begin{alignat}{2}
-\Del u - k^2 u &=f:=\frac{\sin(kr)}{r}  &&\qquad\mbox{in  } \Om,\label{e6.1}\\
\frac{\pa u}{\pa n_\Om} +\i k u &=g &&\qquad\mbox{on } \Ga_R:=\pa\Om.\label{e6.2}
\end{alignat}
Here $\Om$ is the unit regular hexagon with center $(0,0)$ (cf. Figure~\ref{f1})
and $g$ is so chosen that the exact solution is
\begin{equation}\label{e6.3}
u=\frac{\cos(kr)}{k}-\frac{\cos k+\i\sin k}{k\big(J_0(k)+\i J_1(k)\big)}J_0(kr)
\end{equation}
in polar coordinates, where $J_\nu(z)$ are Bessel functions of the first kind.
\begin{figure}[ht]
\centerline{\includegraphics[scale=0.8]{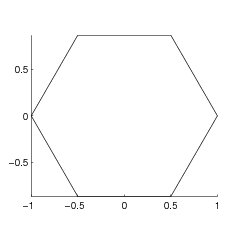}\includegraphics[scale=0.8]{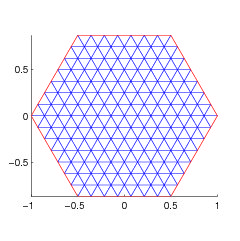}}
\caption{Geometry (left) and a sample mesh $\cT_{1/7}$ that consists of
congruent and equilateral triangles of size $h=1/7$ (right) for
Example 1.}\label{f1}
\end{figure}

For any positive integer $m$, let $\cT_{1/m}$ denote the regular triangulation
that consists of $6m^2$ congruent and equilateral triangles of size $h=1/m$.
See Figure~\ref{f1} (right) for a sample triangulation $\cT_{1/7}$.

\subsection{Stability}\label{ssec-1}
In this subsection, we use the following penalty parameters for the
symmetric IPDG method (cf. \eqref{edg}) according to
Theorem~\ref{thm_sta_uniform} (or \ref{thm main 2}):
\begin{equation}\label{e6.4}
\ga_{1,e}=0.1, \quad \ga_{0,e}=(k^2h)^{2/3}\ga_1^{1/3},
\quad\text{ and }\be_{1,e}=1.
\end{equation}

Given a triangulation $\cT_h$, let $u_h^{\rm FEM}$ be the $P_1$-conforming
finite element approximation of the problem \eqref{e6.1}--\eqref{e6.2}.
Recall that $u_h$ denotes the IPDG solution. Figure~\ref{f1a} plots
the $H^1$-seminorm of the IPDG solution $\norm{u_h}_{1,h}$, the $H^1$-seminorm
of the finite element solution
$\abs{u_h^{\rm FEM}}_{H^1(\Om)}$ for $h=0.05$ and $0.005$, respectively,
and the $H^1$-seminorm of the exact solution $\abs{u}_{H^1(\Om)}$ for
$k=1, \cdots, 230.$ It is shown that
\begin{equation}\label{e6.5}
 \abs{u}_{H^1(\Om)}\simeq 1, \qquad \norm{u_h}_{1,h}\lesssim 1,
\qquad \abs{u_h^{\rm FEM}}_{H^1(\Om)}\lesssim 1.
\end{equation}
We notice that the stability estimate $\norm{u_h}_{1,h}\lesssim 1$ implies
that $\norml{u_h}{\Om}\lesssim 1/k$. These stability estimates are better
than those given by Theorem~\ref{thm_sta_uniform}.

\begin{figure}[ht]
\centerline{\includegraphics[scale=0.8]{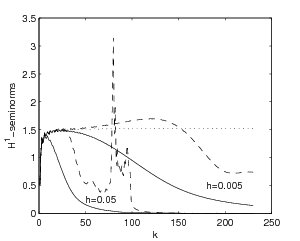}}
\caption{$\norm{u_h}_{1,h}$ (solid), $\abs{u_h^{\rm FEM}}_{H^1(\Om)}$  (dashed)
for $h=0.05$ and $0.005$, respectively. The dotted line gives the
$H^1$-seminorm of the exact solution $\abs{u}_{H^1(\Om)}$.}\label{f1a}
\end{figure}

\subsection{Error of the finite element interpolation}
Given a triangulation $\cT_h$, let  $\hat{u}_h$ be the $P_1$-conforming
finite element interpolation of $u$ on $\cT_h$. Consider in
Figure~\ref{f2} log-log plots of the relative error
$\hat{e}(h,k):=\abs{u-\hat{u}_h}_1\big/\abs{u}_1$ of the finite
element interpolation in $H^1$-seminorm for different $k$ versus $1/h$.
Similar to the 1-D case \cite{ib95a}, All error curves decay with constant
slope of $-1$. Note that the error stays at around 100\% on coarse mesh and
starts to decrease at a certain mesh size. We are interested in the mesh size
where the descent starts. Similar to ``the critical number of degrees of
freedom"  introduced in \cite{ib95a}, we introduce the following
definition of \emph{critical mesh size}.
\begin{definition}\label{def2}
Define---for any fixed $k$ and $f$---the critical mesh size as maximum
mesh size $H(k,f)$ for which
\begin{remunerate}
\item $\tilde{e}(h,k)<1$ for $h<H(k,f)$, and
\item $\tilde{e}(h,k)\to 0$  as $h\to 0.$
\end{remunerate}
\end{definition}
\begin{figure}[ht]
\centerline{\includegraphics[scale=0.8]{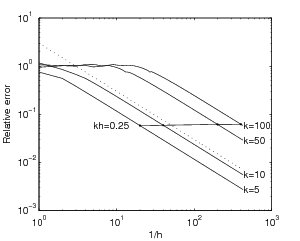}}
\caption{Relative error of the finite element interpolation in $H^1$-seminorm
for $k=5, k=10, k=50,$ and $k=100$. The dotted line gives the reference line
with slope $-1$.}\label{f2}
\end{figure}

Recall that the critical mesh size for the one dimensional case is one
half of the wavelength, that is,  $\frac{\pi}{k}$ \cite{ib95a}. Since the
solution $u$ is axial symmetric (cf. \eqref{e6.3}) and the trace along
any direction may be resolved by a mesh with mesh size less
than $\frac{\pi}{k}$, the critical mesh size  for the finite
element interpolation should be greater than or equal to $\frac{\pi}{k}$.
Figure~\ref{f3} plots the reciprocal of the critical mesh size for the
finite element interpolation computed for all integer $k$ from $1$ to $230$ and
the line passing through the origin has slope $\frac{1}{\pi\sqrt{3}}$.
It shows that
\begin{equation}\label{emeshsizei}
\text{ the critical mesh size for } \hat{u}_h\approx \frac{\sqrt{3}\pi}{k}.
\end{equation}

\begin{figure}[ht]
\centerline{\includegraphics[scale=0.8]{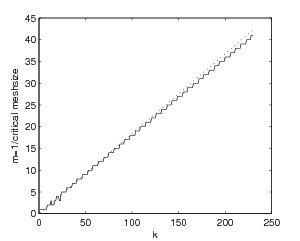}}
\caption{Reciprocal of the critical mesh size of the relative error of the
finite element interpolation in $H^1$-seminorm computed
for $k=1, \cdots, 230$. The dotted line gives the line through the origin
with slope $\frac{1}{\pi\sqrt{3}}$.}\label{f3}
\end{figure}

Figure~\ref{f2} also shows that the error of the finite element
interpolation is controlled by the magnitude $k h$. For illustration,
the points that are computed from $k h=0.25$ are connected. The connecting
line does neither increase nor decrease significantly with the change of $k$.
For more detailed observation, the relative errors of the finite element
interpolations, computed for all integer $k$ from $1$ to $230$  for
$k h=1$ and $k h=0.5$, are plotted in Figure~\ref{f4}. The error for
$k h=1$ stays around $0.247$ and the error for $k h=0.25$ stays around
$0.124$. Note that $0.124/0.247\approx 0.5$ which verifies that the
relative error of the finite element interpolation in $H^1$-seminorm
satisfies $\hat{e}(h,k)=O(k h).$
\begin{figure}[ht]
\centerline{\includegraphics[scale=0.8]{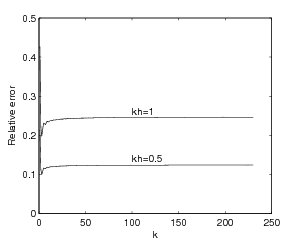}}
\caption{Relative errors of the finite element interpolations in
$H^1$-seminorm computed for $k=1, \cdots, 230$ with mesh size $h$
determined by $kh=1$ and $kh=0.5$, respectively.}\label{f4}
\end{figure}

\subsection{Error of the DG solution}
From Theorem~\ref{thm_main}, the stability estimates in Subsection~\ref{ssec-1}
suggest that the error of the IPDG solution in $H^1$-seminorm could be bounded by
\begin{equation}\label{e6.6}
\abs{u-u_h}_{1,h}\le (1+kh)\big(C_1kh+C_2k^3h^2\big)
\end{equation}
for some constants $C_1$ and $C_2$. The second term on the right hand
side is the so-called pollution error. We now present numerical results 
which verify the above error bound.

In Figure~\ref{f5}, the relative error of the IPDG solution with parameters
given by \eqref{e6.4} and the relative error of the finite element
interpolation are displayed in one plot. The relative error of the
IPDG solution stays around $100\%$ before a critical mesh size is reached,
then decays slowly on a range increasing with $k$, and then decays at a
rate greater than $-1$ in the log-log scale but converges as fast as the
finite element interpolation (with slope $-1$) for small $h$. The relative
error grows with $k$ along line $k h=0.25.$ Unlike the error of the finite
element interpolation, the error of the IPDG solution is not controlled by
the magnitude of $k h$ --- see also Figure~\ref{f6}.

\begin{figure}[ht]
\centerline{\includegraphics[scale=0.8]{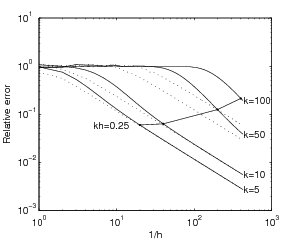}}
\caption{Relative error of the IPDG solution  with parameters
given by \eqref{e6.4} (solid) and relative error of the finite element
interpolation (dotted) in $H^1$-seminorm for $k=5, k=10, k=50,$ and $k=100$,
respectively. }\label{f5}
\end{figure}
\begin{figure}[ht]
\centerline{\includegraphics[scale=0.8]{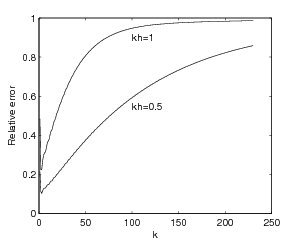}}
\caption{Relative errors of the IPDG solutions  with parameters given
by \eqref{e6.4} in $H^1$-seminorm computed for $k=1, \cdots, 230$
with mesh size $h$  determined by $kh=1$ and  $kh=0.5$, respectively.}\label{f6}
\end{figure}

Figure~\ref{f7} plots the relative error of the IPDG solution with parameters
given by \eqref{e6.4} for $k=5^{2/3}, 10^{2/3}, \cdots, 500^{2/3}$ and $h$
determined by $k^3h^2=1$. The error does not increase with respect to $k$
which verifies \eqref{e6.6}.
\begin{figure}[ht]
\centerline{\includegraphics[scale=0.8]{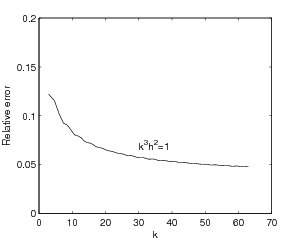}}
\caption{Relative errors of the IPDG solutions with parameters given
by \eqref{e6.4} in $H^1$-seminorm computed for $k=5^{2/3}, 10^{2/3},
\cdots, 500^{2/3}$ with mesh size $h$  determined by $k^3h^2=1$.}\label{f7}
\end{figure}

Figure~\ref{f8} plots the reciprocal of the critical mesh size for the
IPDG solution with parameters given by \eqref{e6.4} computed for all integer
$k$ from $1$ to $230$ and the lines passing through the origin have slopes
$\frac{1}{1.35\pi}$ and $\frac{1}{\pi}$, respectively. It shows that
\begin{equation}\label{emeshsizedg}
 \text{ the critical mesh size for } u_h\approx \frac{1.35\pi}{k}
\end{equation}
with two exceptions but they are still less than $\frac{\pi}{k}.$
It is interesting that the dependence on $1/k$ is essentially linear.
We consider the IPDG solution  with parameters given by \eqref{e6.4}
for $k=100$ on the mesh with mesh size $h=1/60$. The relative error
in $H^1$-seminorm is about 0.9898. Figure~\ref{f8a} presents the surface
plots of the interpolation (left) and the IPDG solution (right).
It is shown that the IPDG solution has a correct shape although
its amplitude is not very accurate.

\begin{figure}[ht]
\centerline{\includegraphics[scale=0.8]{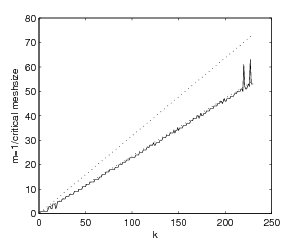}}
\caption{Reciprocal of the critical mesh size of the relative error of the
IPDG solution  with parameters given by \eqref{e6.4} in $H^1$-seminorm
computed for $k=1, \cdots, 230$. The dotted lines give the lines
through the origin with slopes $\frac{1}{1.35\pi}$ and $\frac{1}{\pi}$.}\label{f8}
\end{figure}

\begin{figure}[ht]
\centerline{\includegraphics[scale=0.8]{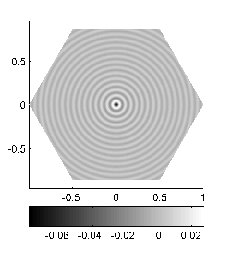}\includegraphics[scale=0.8]{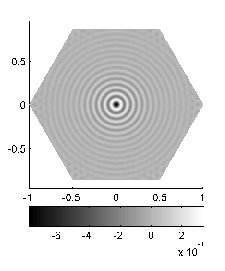}}
\caption{Surface plots of the interpolation (left) and the IPDG
solution (right)  with parameters given by \eqref{e6.4} for $k=100$
on the mesh with mesh size $h=1/60$.}\label{f8a}
\end{figure}

\subsection{Sensitivity of the error bounds with respect to penalty parameters}
\label{sec-6.4}
In this subsection, we examine the sensitivity of the error of the
IPDG solution in $H^1$-seminorm with respect to the
parameters $\ga_{0,e}$, $\be_{1,e}$, and $\ga_{1,e}$, respectively.

First, we examine the sensitivity of the error in $\ga_{0,e}$.
To the end, for $k=5$ and $k=50$, respectively, we
fix $\ga_{1,e}=0.1$ and $\be_{1,e}=1$,
and compute the IPDG solution with the following $\ga_{0,e}$:
$\ga_{0,e}=(k^2h)^{2/3}\ga_{1,e}^{1/3}$
(see \eqref{e6.4}), $\ga_{0,e}=1$, $\ga_{0,e}=0.01$, and $\ga_{0,e}=100$.
Figure~\ref{f9} plots the relative error of the IPDG solution
for each run. We observe that the error in the $H^1$-seminorm is
not sensitive with respect to the parameter $\ga_{0,e}$. It is
clear that $\ga_{0,e}$ affects the continuity of the solution.
The larger the parameter $\ga_{0,e}$, the more continuous the IPDG solution.
\begin{figure}[ht]
\centerline{\includegraphics[scale=0.8]{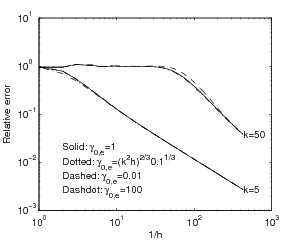}}
\caption{Relative error of the IPDG solution with parameters
$\ga_{1,e}=0.1, \be_{1,e}=1$, and each of following $\ga_{0,e}$:
$\ga_{0,e}=(k^2h)^{2/3}\ga_{1,e}^{1/3}$ (dotted), $\ga_{0,e}=1$ (solid),
$\ga_{0,e}=0.01$ (dashed), and $\ga_{0,e}=100$ (dashdot) in the
$H^1$-seminorm for $k=5$ and $k=50$, respectively. }\label{f9}
\end{figure}

Secondly, we test the sensitivity of the error in $\be_{1,e}$.
Figure~\ref{f9a} plots the relative error in the $H^1$-seminorm of the IPDG
solution with parameters $\ga_{0,e}=1, \ga_{1,e}=0.1$, and each of the
following $\be_{1,e}$: $\be_{1,e}=0, 1, 100$, for $k=5$ and $k=50$,
respectively. Again, we observe that the error in the $H^1$-seminorm
is not sensitive with respect to the parameter $\be_{1,e}$.
\begin{figure}[ht]
\centerline{\includegraphics[scale=0.8]{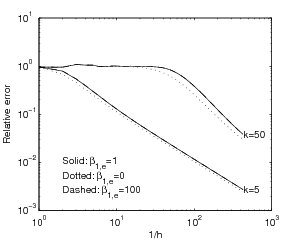}}
\caption{Relative error in the $H^1$-seminorm of the IPDG solution with parameters
$\ga_{0,e}=1, \ga_{1,e}=0.1$, and each of the following $\be_{1,e}$:
$\be_{1,e}=0$ (dotted), $\be_{1,e}=1$ (solid), and $\be_{1,e}=100$ (dashed)
for $k=5$ and $k=50$, respectively. }\label{f9a}
\end{figure}

Finally, we examine the sensitivity of the error in $\ga_{1,e}$.
To the end, we fix $\ga_{0,e}=1$ and $\be_{1,e}=1$ and compute the IPDG
solution with the following $\ga_{1,e}$:
$\ga_{1,e}=0, 0.01, 0.1, 1$ for $k=5$ and $k=50$, respectively.
Figure~\ref{f10} plots the relative error in the $H^1$-seminorm
of the IPDG solution for each run. We observe that the error has
a similar behavior as the error of the
finite element solution for small value $\ga_{1,e}=0, 0.01$
(cf. Figure~\ref{f13}) and the solution is more stable for
larger $\ga_{1,e}$, but a large $\ga_{1,e}$, say $\ga_{1,e}=1$,
may result in a large (absolute) error.
\begin{figure}[ht]
\centerline{\includegraphics[scale=0.8]{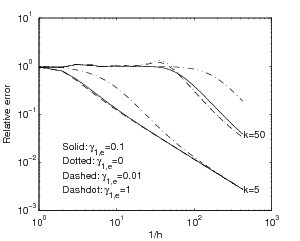}}
\caption{Relative error in the $H^1$-seminorm of the IPDG solution with parameters
$\ga_{0,e}=1, \be_{1,e}=1$, and each of the following $\ga_{1,e}$:
$\ga_{1,e}=0.1$ (solid), $\ga_{1,e}=0$ (dotted),  $\ga_{1,e}=0.01$ (dashed),
and $\ga_{1,e}=1$ (dashdot) for $k=5$ and $k=50$, respectively. }\label{f10}
\end{figure}
\subsection{Reduction of the pollution effect}\label{sec-6.5}
One advantage of the IPDG method is that it contains several parameters
which can be tuned for a particular purpose. In \cite{aldr06}, it is
shown that it is possible to reduce the pollution error of the IPDG method
by choosing appropriate parameters $\si$ and $\i\ga_{0,e}$. Recall that
all choice of $\si$ but one lead to non-symmetric formulations. In this
subsection, we shall show that appropriate choice of the parameter
$\ga_{1,e}$ can significantly reduce the pollution error of the symmetric
IPDG method (with $\si=1$). We use the following parameters:
\begin{equation}\label{e6.7}
    \i\ga_{1,e}=-0.07+0.01\i, \quad \ga_{0,e}=100, \quad\text{ and }\be_{1,e}=1.
\end{equation}
We remark that $\i\ga_{1,e}$ is simply chosen from the set
$\set{0.01(p+q\i), -50\le p,q\le 50}$ to minimize the relative error of the
IPDG solution in $H^1$-seminorm with $\ga_{0,e}=100$ and $\be_{1,e}=1$ for
wave number $k=50$ and mesh size $h=1/20.$

In Figure~\ref{f11}, the relative error of the IPDG solution with parameters
given by \eqref{e6.7} and the relative error of the finite element interpolation
are displayed in one plot. The IPDG method with  parameters given
by \eqref{e6.7} is much better than the IPDG method using  parameters given
by \eqref{e6.4} (cf. Figure~\ref{f5}). The relative error does neither
increase nor decrease significantly with the change of $k$ along line
$k h=0.25$ for $k\le 100$. But this does not mean that the pollution error
has been eliminated.
\begin{figure}[ht]
\centerline{\includegraphics[scale=0.8]{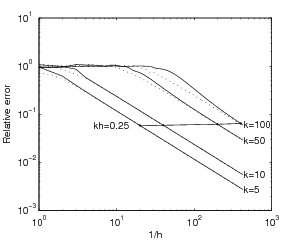}}
\caption{Relative error of the IPDG solution  with parameters given
by \eqref{e6.7} (solid) and relative error of the finite element
interpolation (dotted) in $H^1$-seminorm for $k=5, k=10, k=50,$
and $k=100$, respectively. }\label{f11}
\end{figure}
For more detailed observation, the relative errors of the IPDG solution
with parameters given by \eqref{e6.7}, computed for all
integer $k$ from $1$ to $230$  for $k h=1$ and $k h=0.5$, are plotted in
Figure~\ref{f12}. It is shown that the pollution error is
reduced significantly (cf. Figure~\ref{f6} and Figure~\ref{f4}).
\begin{figure}[ht]
\centerline{\includegraphics[scale=0.8]{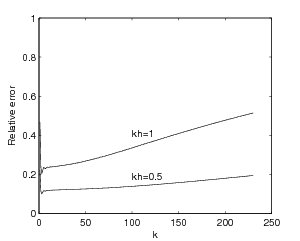}}
\caption{Relative errors of the IPDG solution with parameters given
by \eqref{e6.7} in $H^1$-seminorm computed for $k=1, \cdots, 230$
with mesh size $h$ determined by $kh=1$ and $kh=0.5$, respectively.}\label{f12}
\end{figure}

\subsection{Comparison between the IPDG solution and the finite element solution}
We have shown the flexibility and performance of the IPDG method in previous
subsections. In this subsection, we give a comparison between the IPDG method
and the finite element method. One disadvantage of the IPDG method compared
to the finite element method is that the linear system of the IPDG
discretization involves more number of degrees of freedom than that of
finite element discretization on the same mesh. In two dimensional case
it is about three times more. So in the asymptotic range, the IPDG method
is less effective in terms of number of degrees of freedom.
 We shall show that, for Problem \eqref{e6.1}--\eqref{e6.2}, the
IPDG solution is more stable than the finite solution for large $h$,
and it is  possible to choose appropriate parameters such that the
IPDG method is more effective than the finite element method in preasymptotic
range even in terms of number of degrees of freedom.

In Figure~\ref{f13}, the relative error of the finite element solution
and the relative error of the finite element interpolation are displayed
in one plot. The relative error of the finite element solution first
oscillates around $100\%$, then decays at a rate greater than $-1$ in
the log-log scale but converges as fast as the finite element interpolation
(with slope $-1$) for small $h$. The relative error grows with $k$ along
line $k h=0.25.$ The error of the finite element solution is not controlled by
the magnitude of $k h$ --- see also Figure~\ref{f14}.
\begin{figure}[ht]
\centerline{\includegraphics[scale=0.8]{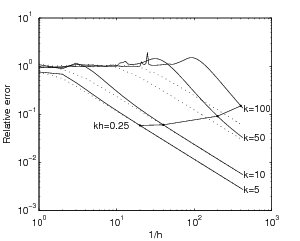}}
\caption{Relative error of the finite element solution  (solid) and relative
error of the finite element interpolation (dotted) in $H^1$-seminorm
for $k=5, k=10, k=50,$ and $k=100$, respectively.}\label{f13}
\end{figure}

\begin{figure}[ht]
\centerline{\includegraphics[scale=0.8]{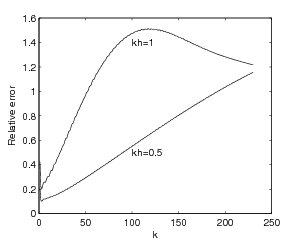}}
\caption{Relative errors of the finite element solutions in $H^1$-seminorm
computed for $k=1, \cdots, 230$ with mesh size $h$ determined by $kh=1$
and $kh=0.5$, respectively.}\label{f14}
\end{figure}

Figure~\ref{f15} plots the reciprocal of the critical mesh size for the
finite element solution computed for all $k$ from $1$ to $230$ and the
curve  $m=\sqrt{k^3/48}$. It is shown that
\begin{equation}\label{emeshsizefem}
\text{ the critical mesh size for } u_h^{\rm FEM}\approx \sqrt{\frac{48}{k^3}}.
\end{equation}

\begin{figure}[ht]
\centerline{\includegraphics[scale=0.8]{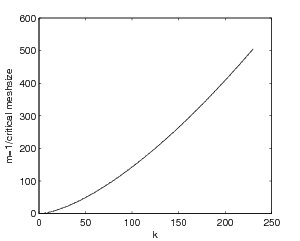}}
\caption{Reciprocal of the critical mesh size of the relative error of the
finite element solution in $H^1$-seminorm computed for $k=1, \cdots, 230$.
The dotted line gives the curve  $ m=\sqrt{k^3/48}$.}\label{f15}
\end{figure}

We can see that the IPDG solution is more stable than  the finite element
solution.  For more detailed comparison, we consider the problem
\eqref{e6.1}--\eqref{e6.2} with wave number $k=100$.
The traces of the IPDG solutions with parameters given by \eqref{e6.7}
and the finite element solutions in the $xz$-plane for mesh sizes
$h=1/50, 1/120$, and $1/200$, and the trace of the exact
solution in the $xz$-plane, are plotted in Figure~\ref{f16}. The shape
of the IPDG solution is roughly same as that of the exact solution
for $h=1/50,$. They matches very well for $h=1/120$ and even
better for $h=1/200$. While the finite element solution has a wrong
shape near the origin for $h=1/50$ and $h=1/120$ and only has a correct
shape for $h=1/200$. The phase error appears in all the three cases for
the finite element solution.
\begin{figure}[ht]
\centerline{\includegraphics[scale=0.25]{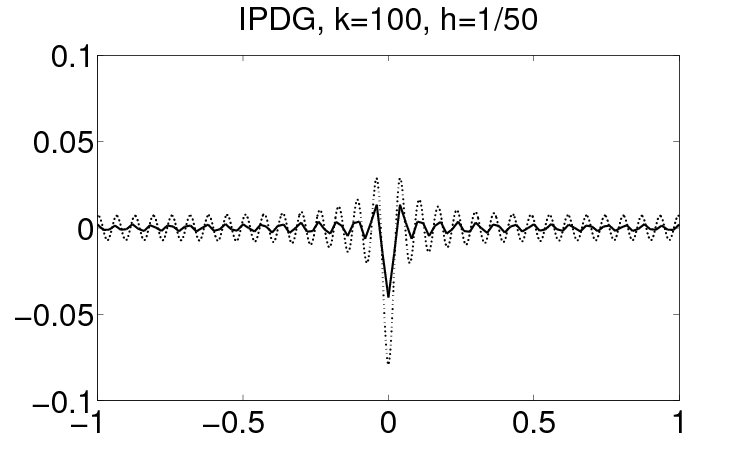}\includegraphics[scale=0.25]{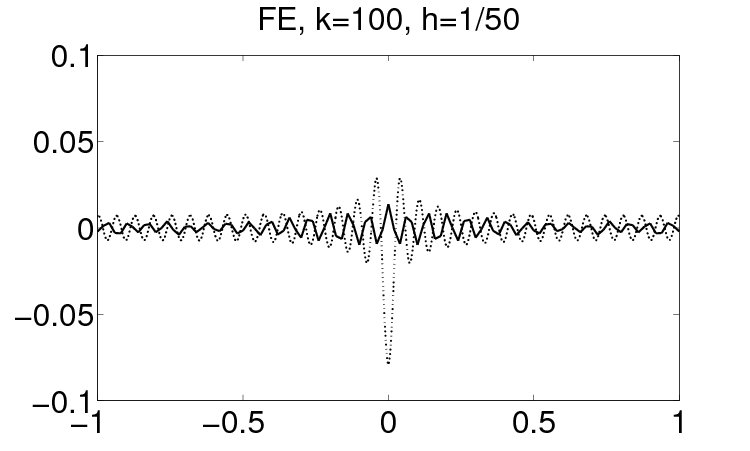}}
\centerline{\includegraphics[scale=0.25]{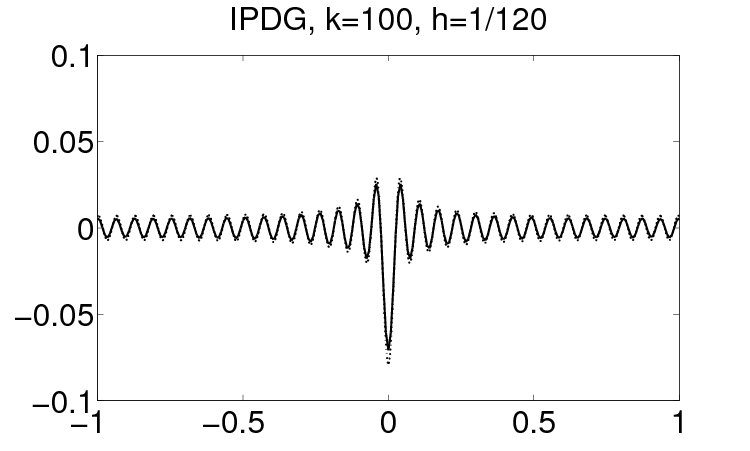}\includegraphics[scale=0.25]{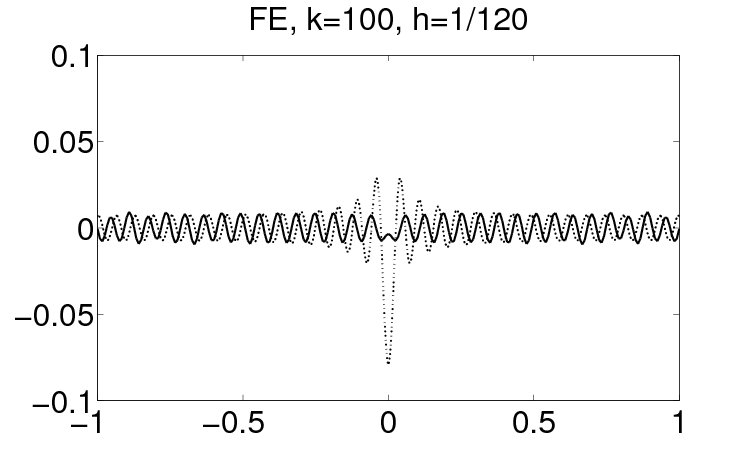}}
\centerline{\includegraphics[scale=0.25]{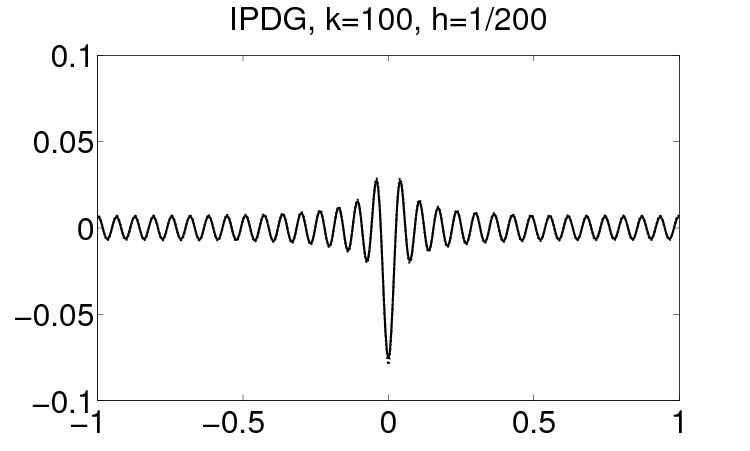}\includegraphics[scale=0.25]{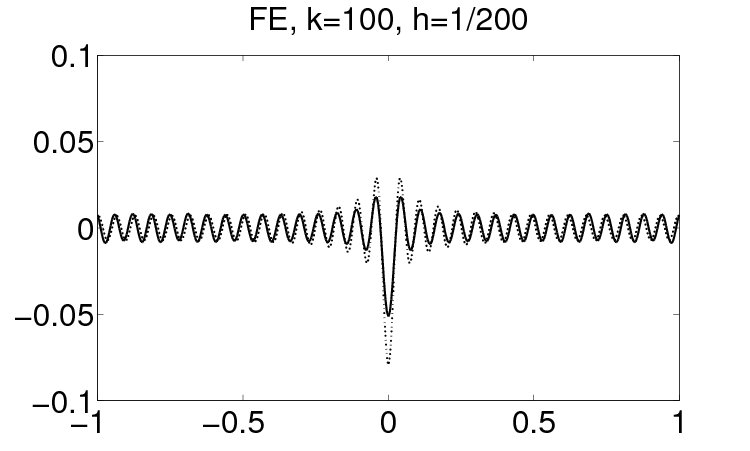}}
\caption{The traces of the IPDG solutions (left) with parameters given by
\eqref{e6.7}  and the finite element solutions (right) in the $xz$-plane
for $k=100$ and mesh sizes $h=1/50, 1/120$, and $1/200$, respectively.
The dotted lines give the trace of the exact solution in  the $xz$-plane.}
\label{f16}
\end{figure}

Table~\ref{tab1} shows the numbers of degrees of freedom needed for $30$\%
relative errors in $H^1$-seminorm for the finite element  interpolation, the
IPDG solution with parameters given by \eqref{e6.7}, and the finite element
solution, respectively.  The finite element method needs less DOFs when $k=10$
and $k=50$ than the IPDG method does, but the situation reverses when
$k=100$ and $k=200$.
\begin{table}[ht]
\centering
\begin{tabular}{|l|c|c|c|c|}
  \hline
    k & 10 & 50 & 100 & 200 \\\hline
    Interpolation & 217 (1/8)& 5,167 (1/41) & 20,419 (1/82) & 81,181 (1/164) \\\hline
    IPDG & 1,152 (1/8)& 38,088 (1/46) & 217,800 (1/110) & 1,431,432 (1/282) \\\hline
    FEM & 397 (1/11)& 30,301 (1/100) & 229,357 (1/276) & 1,804,201 (1/775) \\\hline
  \end{tabular}
\caption{Numbers of degrees of freedom needed for 
30\% relative errors
in $H^1$-seminorm for the finite element interpolation, the IPDG solution
with parameters given by \eqref{e6.7}, and the finite element solution,
respectively. The fractions in the parentheses give the corresponding
mesh sizes.}\label{tab1}
 \end{table}


\end{document}